\documentclass[letterpaper, 10 pt, conference]{ieeeconf}  

\IEEEoverridecommandlockouts 

\usepackage{cite}

\usepackage{empheq}
\usepackage{mwe}

\usepackage{graphicx}
\usepackage{xcolor}
\usepackage{epstopdf}
\usepackage{fancyhdr}
\usepackage{subcaption}

\usepackage{amsmath,amsthm,amssymb,mathtools,bbm,dsfont,physics,aligned-overset}
\usepackage[T1]{fontenc}
\usepackage{color}
\usepackage[geometry]{ifsym}
\usepackage{siunitx}
\usepackage{multirow}
\usepackage{placeins}
\usepackage{booktabs}
\usepackage{rotating}
\usepackage{array}
\usepackage{dsfont}
\usepackage{bbm}
\usepackage{bbold}
\usepackage{caption,subcaption}
\usepackage[hyphens]{url}

\usepackage[capitalize]{cleveref}
\usepackage{comment}

\usepackage{algorithm}
\usepackage{algpseudocode}

\usepackage{epsfig}
\usepackage{epstopdf}

\usepackage{xcolor}

\DeclareMathOperator{\Vc}{\mathcal{V}}
\DeclareMathOperator{\Wc}{\mathcal{W}}
\DeclareMathOperator{\Rc}{\mathcal{R}}
\DeclareMathOperator{\Tc}{\mathcal{T}}

\DeclareMathOperator{\R}{\mathbb{R}}

\newcommand{\bu}{\mathbf{u}}
\newcommand{\bx}{\mathbf{x}}
\newcommand{\bw}{\mathbf{w}}
\newcommand{\bv}{\mathbf{v}}
\newcommand{\by}{\mathbf{y}}
\newcommand{\bQ}{\mathbf{Q}}
\newcommand{\bR}{\mathbf{R}}
\newcommand{\bC}{\mathbf{C}}
\newcommand{\bH}{\mathbf{H}}
\newcommand{\bG}{\mathbf{G}}
\newcommand{\bD}{\mathbf{D}}
\newcommand{\bq}{\mathbf{q}}
\newcommand{\bU}{\mathbf{U}}
\newcommand{\bW}{\mathbf{W}}
\newcommand{\bV}{\mathbf{V}}
\newcommand{\bF}{\mathbf{F}}
\newcommand{\bM}{\mathbf{M}}

\usepackage{bm}
\newcommand{\boldeta}{\bm{\eta}}
\newcommand{\bSigma}{\bm{\Sigma}}

\definecolor{purple}{rgb}{0.62, 0.0, 0.77}
\newcommand{\luci}{\color{black}} 
\definecolor{dartmouthgreen}{rgb}{0.05, 0.5, 0.06}
\newcommand{\marti}{\color{black}} 
\definecolor{cobalt}{rgb}{0.02, 0.29, 0.65}
\newcommand{\agu}{\color{black}} 

\usepackage[bottom]{footmisc}
\usepackage{graphicx}

\DeclareMathOperator*{\argmin}{argmin}   
  
\usepackage{mathtools}
\usepackage[acronym]{glossaries}

\newtheorem{theorem}{Theorem}
\newtheorem{lemma}[theorem]{Lemma}

\newtheorem{proposition}[theorem]{Proposition}

\theoremstyle{definition}

\newtheorem{example}{Example}
\newtheorem{remark}{Remark}

\usepackage{stmaryrd}

\usepackage{tikz}
\usepackage{lipsum}

\let\argmin\relax
\let\argmax\relax
\DeclareMathOperator*{\argmax}{arg\,max}
\DeclareMathOperator*{\argmin}{arg\,min}

\title{\LARGE \bf 
Distributionally Robust LQG with Kullback-Leibler Ambiguity Sets}

\author{
Marta Fochesato, Lucia Falconi, Mattia Zorzi, Augusto Ferrante, John Lygeros%
\thanks{
Marta Fochesato, Lucia Falconi, and John Lygeros are with the Automatic Control Laboratory, ETH Z\"urich, 8092 Z\"urich, Switzerland (e-mail: \texttt{mfochesato@ethz.ch, lfalconi@ethz.ch, jlygeros@ethz.ch}). Mattia Zorzi and Augusto Ferrante are with the Department of Information Engineering, University of Padova, Padova, Italy (e-mail: \texttt{zorzimat@dei.unipd.it, augusto@dei.unipd.it}).
}%
}

\begin{document}

\maketitle
\thispagestyle{empty}
\pagestyle{empty}

\IEEEpeerreviewmaketitle

\begin{abstract}                          
    The Linear Quadratic Gaussian (LQG) controller is known to be inherently fragile to model misspecifications common in real-world situations. We consider discrete-time partially observable stochastic linear systems and provide a robustification of the standard LQG against distributional uncertainties on the process and measurement noise. Our distributionally robust formulation specifies the admissible perturbations by defining a relative entropy based ambiguity set individually for each time step along a finite-horizon trajectory, and minimizes the worst-case cost across all admissible distributions. We prove that the optimal control policy is still linear, as in standard LQG, and derive a computational scheme grounded on iterative best response that provably converges to the set of saddle points. Finally, we consider the case of endogenous uncertainty captured via decision-dependent ambiguity sets and we propose an approximation scheme based on dynamic programming.
    
    \end{abstract}

    \section{Introduction}
    The Linear Quadratic Regulator (LQG) is a cornerstone of control theory, and has been applied across several domains ranging from engineering, to economics, to computer science. It involves controlling a linear system subject to additive disturbances via the design of a control policy that minimizes a quadratic cost function. For linear systems, quadratic costs and additive Gaussian noise, it is well-known that the resulting optimal policy is linear in the observations. However, in practice, the system dynamics are often only approximately known, and the noise is not necessarily Gaussian. This leads to a mismatch between the nominal model and the true system, which can adversely affect performances \cite{doyle1978guaranteed}. 
    
    We consider a generalization of the finite-horizon LQG framework for discrete-time linear stochastic systems subject to model mismatch. We interpret the problem as a zero-sum game between the controller and a distribution chosen from an ambiguity set that may change at each time step. Each ambiguity set is represented as a ball defined in the Kullback-Leibler (KL) divergence centered at a nominal  Gaussian distribution. The goal of the control designer is then to synthesize a control policy that minimizes the worst-case expected cost. We refer to this problem as the \textit{Distributionally Robust LQG} (DR-LQG) problem with KL ambiguity sets. 

    For the DR-LGQ problem, we show that, even under distributional ambiguity, the optimal control policy is linear in the observations, as in the standard LQG. Following \cite{taskesen2023distributionally,lanzetti2024optimality} we re-parametrize the control policy in terms of purified observations and propose a novel "sandwich" argument to show that the DR-LQG is optimally solved by a linear policy and that the worst-case distribution is still Gaussian. We show that the resulting game admits a Nash equilibrium and provide a numerical scheme based on regularized best response that converges linearly to the set of saddle points of the DR-LQG. The best responses admits closed-form expression, making the algorithm computationally attractive. We further consider the case of endogenous uncertainty captured via decision-dependent ambiguity sets, well suited to capture perturbations in the system matrices. We derive an approximated closed-form recursion based on dynamic programming, and use it within a coordinate gradient descent scheme to solve the problem with linear convergence rate guarantees. Finally, we provide numerical simulations to illustrate the effectiveness of our approaches.

    \subsubsection*{Literature review}
    Traditionally, the optimal control of systems subject to uncertainty has been studied through the LQR/LQG theory by synthesizing policies that either minimize a quadratic functional in the state and control action pair, or the $\mathcal{H}_2$ norm of the system's transfer function \cite{hassibi1999indefinite}. Motivated by the fragility of the LQG against model mismatch \cite{doyle1978guaranteed}, several robustification approaches were suggested in the literature. For example, if the perturbation is modeled as adversarial with \textit{known} finite energy bound, the  $\mathcal{H}_\infty$ synthesis approach can be used to minimize the infinity norm of the transfer function
    mapping disturbances to a measurable performance metric \cite{zhou1998essentials}. While stability is guaranteed under any allowable process noise signal, the $\mathcal{H}_\infty$ controller might result in conservative behaviors as it plans for the worst uncertainty. This motivated the development of combined $\mathcal{H}_2 / \mathcal{H}_\infty$ strategies, trying to alleviate part of the conservatism. A different approach is represented by risk-sensitive control that incorporates risk sensitivity into the control design task by minimizing the entropic risk measure of the cost, rather than its expected value \cite{jacobson1973optimal,whittle1981risk}. 

    Distributional robust control (DRC) is an alternative paradigm that is gaining momentum \cite{van2015distributionally,hakobyan2022wasserstein,fochesato2022data,aolaritei2023wasserstein,petersen2000minimax,hakobyan2024wasserstein,brouillon2023distributionally,falconi2025distributionally}. The DRC problem seeks a control policy that minimizes the expected cost under the worst-case noise distribution in an ambiguity set, i.e., a set of probability distributions among which we can reasonably expect to find the true noise distribution. DRC robustifies against model misspecifications in the space of probabilities. Different types of ambiguity sets have been proposed in the literature based on moment constraints \cite{van2015distributionally}, total variation distance \cite{tzortzis2015dynamic}, Wasserstein distance \cite{aolaritei2023wasserstein} (or, more broadly, optimal transport distances), and $\phi-$divergences \cite{petersen2000minimax}. 

    In \cite{taskesen2023distributionally} a DR-LQG problem similar to the one considered here is solved by considering ambiguity sets defined on the basis of the Wasserstein distance. Compared to \cite{taskesen2023distributionally}, (i) we motivate the choice of the KL distance by providing interesting connections with system identification and risk measure theory, (ii) we prove optimality of the linear policies for the DR-LQG problem with KL ambiguity sets, and (iii) we provide a novel best response algorithm that exploits the availability of closed-form expressions for the best responses, unlike the gradient-based method suggested in \cite{taskesen2023distributionally} that relies on automatic differentiation. 
    
  The dynamic-programming based recursion for decision-dependent ambiguity sets, is related to the approaches in \cite{hakobyan2024wasserstein} and \cite{hakobyan2022wasserstein}. These works address a regularized relaxed problem with a fixed penalty value in the objective function; hence, there are no guarantees that the worst-case  distribution, against which the controller is hedging, belongs to the specified ambiguity set. On the contrary, our formulation addresses the exact constrained formulation. Additionally, in \cite{hakobyan2024wasserstein} and \cite{hakobyan2022wasserstein} distributed uncertainty modeling is not considered. Finally, in \cite{hakobyan2024wasserstein} and \cite{hakobyan2022wasserstein} the problem is solved by neglecting a crucial non-linear dependence in the optimality equations; we tackle this problem by linearizing this dependence (instead of disregarding it) with obvious advantages in terms of performances, as demonstrated in the numerical section.

    \subsubsection*{Outline} 
In Section II we present the necessary background, while in Section III we describe the problem statement. In Section IV the theoretical properties of the DR-LQG are analyzed and in Section V a computational framework is presented. Section VI considers the case with endogenous ambiguity sets, while in Section VII simulation results are reported. Finally, we draw our conclusions in Section VIII. We relegate all proofs to the appendix.

    \section{Preliminaries}\label{sec:preliminaries}
   
        \subsubsection*{Notation} $\mathbb{R}^n, \mathbb{R}^n_+, \mathbb{R}^n_{++}$ denote the set of reals, nonnegative reals, and positive reals, respectively. $\mathbb{S}^n_{+}$ (resp. $\mathbb{S}^n_{++}$) denotes the cone of $n\times n$ symmetric positive semi-definite (resp. positive definite) matrices. For any $A,B \in \mathbb{S}_+^n$, the relation $A\succeq B$ ($A \succ B$) means that $A - B \in \mathbb{S}^n_{+}$($A - B \in \mathbb{S}^n_{++}$). Given a matrix $A$, $A^\top$ denotes its transpose and $|A|$ its determinant. 
    All random object are defined on a probability space $(\Omega, \mathcal{F}, \mathbb{P})$; thus, the distribution of a random vector $\xi : \Omega \rightarrow \mathbb{R}^n$ is given by the pushforward distribution $\mathbb{P}_\xi = \mathbb{P} \circ \xi^{-1}$ of $\mathbb{P}$ with respect to $\xi$. We denote as $\mathcal{P}(\mathbb{R}^n)$ the set of absolutely continuous probability distributions on $\mathbb{R}^n$, and as $\mathcal{P}_\mathcal{G}(\mathbb{R}^n)$ its restriction to the set of Gaussians on $\mathbb{R}^n$. Finally, we denote as $\mathcal{N}(\mu, \Sigma)$ a Gaussian distribution with mean $\mu$ and covariance $\Sigma$.

    \subsubsection*{Relative entropy ambiguity sets} Given any two probability distributions $\mathbb{P}_1, \mathbb{P}_2 \in \mathcal{P}(\mathbb{R}^n)$, a KL-ambiguity set of radius $\rho > 0$ centered at $\mathbb{P}_1$ is
    $$
    \mathcal{B} := \{ \mathbb{P}_2 \in \mathcal{P}(\mathbb{R}^n), \: R(\mathbb{P}_2 || \mathbb{P}_1) \leq \rho\},
    $$
  where $R(\mathbb{P}_2 || \mathbb{P}_1)$ is the KL divergence from distribution $\mathbb{P}_2$ to the centre $\mathbb{P}_1$. For $\mathbb{P}_1, \mathbb{P}_2$ admitting densities $p_1(x), p_2(x)$, respectively, then 
\begin{equation}\label{KL:density}
  R(\mathbb{P}_2 || \mathbb{P}_1) = \int p_2(x) \log \frac{p_2(x)}{p_1(x)} dx.
\end{equation}
We do not differentiate between $\mathbb{P}_i$ and its density $p_i(x)$ in the following, when clear from the context. 

It is known that $R(\mathbb{P}_2 || \mathbb{P}_1)$ is not a distance since it is not symmetric and does not obey the triangular inequality; however, it is a pseudo-distance since it satisfies (i) $R(\mathbb{P}_2 || \mathbb{P}_1) \geq 0$ and (ii) $R(\mathbb{P}_2 || \mathbb{P}_1) = 0$ if and only if $\mathbb{P}_2 = \mathbb{P}_1$. Moreover, for a given $\mathbb{P}_1$, $R(\mathbb{P}_2 || \mathbb{P}_1)$ is a strictly convex function of $\mathbb{P}_2$ on $\mathcal{P}(\mathbb{R}^n)$. When restricted to Gaussian distributions, the KL divergence admits a closed form expression \cite{mackay2003information}: For  $\mathbb{P}_1 \triangleq \mathcal{N}({\mu}_1,{\Sigma}_1)$ and $\mathbb{P}_2 \triangleq \mathcal{N}({\mu}_2,{\Sigma}_2) \in \mathcal{P}_\mathcal{G}(\mathbb{R}^n)$, with $\Sigma_1, {\Sigma}_2 \succ 0$, it reads

    \begin{equation}\label{eq:tightness}
    \begin{aligned}
    R(\mathbb{P}_2 || \mathbb{P}_1)\!  = &  \frac{1}{2}\Big( ({\mu}_2 - \mu_1)^\top {\Sigma}_1^{-1}({\mu}_2 - \mu_1) + \emph{Tr}( {\Sigma}_1^{-1} {\Sigma}_2) \\
    &  - n + \ln \frac{|{\Sigma}_1|}{|{\Sigma}_2|}  \Big).
    \end{aligned}
\end{equation}

    \subsubsection*{Convex optimization} Finally, we recall some results from convex optimization.
    \begin{lemma}[Pointwise maximization \cite{boyd:vandenberghe:2004}, p.~81]\label{lemm:convexity_max}
    Let $X$ be a normed space and $\{h_i(x)|i \in I\}$ be a collection of functions with the same domain $\mathcal{H}$. Assume that $\mathcal{H}$ is a convex subset of $X$ and that  $h_i(x)$ is convex for each $i$. Then
    $
    g(x) \coloneqq \sup_{i \in I}h_i(x)
    $
    is also convex.
    \end{lemma}
    \smallskip
    
    \begin{lemma}[Partial minimization rule \cite{boyd:vandenberghe:2004}, p.~87] \label{lemm:convexity_min}
    Let $h(x,y)$ be jointly convex in $(x,y)$ and $K$ be a convex non-empty set. Then, the function
    $
    g(x) = \inf_{y \in K} h(x,y)
    $
    is convex in $x.$ 
    \end{lemma}

    \section{Problem setup}\label{sec:pb}
    Consider a discrete-time stochastic linear system 
    \begin{equation}\label{eq:nominal:system}
    \begin{aligned}
    x_{t+1} & = Ax_t + Bu_t + w_t,\\
    y_t & = Cx_t + v_t,
    \end{aligned}
    \end{equation}
    where $x_t \in \mathbb{R}^n, u_t \in \mathbb{R}^m, y_t \in \mathbb{R}^p$ are the system state, input and output, respectively, and $A \in \mathbb{R}^{n\times n}, B \in \mathbb{R}^{n \times m}, C \in \mathbb{R}^{p\times n}$ are the corresponding system matrices, that we assume to be time-invariant to streamline the presentation\footnote{The theory remains valid in the case of time-dependant parameters.}. 
    The exogenous random vectors $x_0$, $\{w_t\}_{t =0}^{T-1}$ and $\{v_t\}_{t=0}^{T-1}$, corresponding to the initial condition and the process and measurement noises, are assumed to be mutually independent and follow probability distributions $\mathbb{P}_{x_0}$, $\left\{\mathbb{P}_{w_t}\right\}_{t=0}^{T-1}$, and $\left\{\mathbb{P}_{v_t}\right\}_{t=0}^{T-1}$, respectively. We assume without loss of generality that $\Omega= \mathbb{R}^n \times \mathbb{R}^{nT} \times \mathbb{R}^{pT}$ is the space of realizations of the exogenous uncertainties, $\mathcal{F}$ is the Borel $\sigma$-algebra on $\Omega$ and $\mathbb{P}=\mathbb{P}_{x_0} \times \left(\otimes_{t=0}^{T-1} \mathbb{P}_{w_t}\right) \otimes\left(\otimes_{t=0}^T \mathbb{P}_{v_t}\right)$, where $\mathbb{P}_1 \otimes \mathbb{P}_2$ denotes the independent coupling of the distributions $\mathbb{P}_1$ and $\mathbb{P}_2$.
    Further, for $t \geq 0$, let $\mathcal{F}_t = \sigma (y_{0:t})$ be
    the $\sigma$-algebra generated by all observations up to time $t$, and let
    $\mathcal{F}_{-1}$ be the trivial $\sigma$-algebra. We restrict attention to inputs $u_t$ that are $\mathcal{F}_t$-measurable
    with respect to the observations up to time $t$ and assume that there exists a
    measurable function $\pi_t $ such that $u_t = \pi_t(y_{0:t})$. The control
    policy $\pi \triangleq \{\pi_t\}_{t=0}^{T-1} \in \mathcal{U}_{\by}$ is the collection of all such functions. 
    

    
    
   We depart from the standard LQG framework, and assume that $\mathbb{P}$ is unknown. Instead, one has only access to a nominal model $\hat{\mathbb{P}}$, for example retrieved via statistical analysis of expert knowledge, that is assumed to be of the form $\hat{\mathbb{P}}=\hat{\mathbb{P}}_{x_0} \otimes \left(\otimes_{t=0}^{T-1} \hat{\mathbb{P}}_{w_t}\right) \otimes\left(\otimes_{t=0}^T \hat{\mathbb{P}}_{v_t}\right)$, where $\hat{\mathbb{P}}_{x_0} \triangleq \mathcal{N}(0 , \hat{W}_{-1}), \hat{\mathbb{P}}_{w_t} \triangleq \mathcal{N}(0 , \hat{W}_{t}), \hat{\mathbb{P}}_{v_t} \triangleq \mathcal{N}(0 , \hat{V}_t)$ for $\hat{W}_t, \hat{V}_t \succ 0$. For $Q, Q_T \in \mathbb{S}_{+}^n$ and $R \in \mathbb{S}_{++}^{m}$, our aim is to solve the following minimax problem:

    \begin{equation}\label{eq:DR-LQG}
        \inf_{\pi} \max _{\mathbb{P} \in \mathcal{B}} \mathbb{E}_{\mathbb{P}}\left[\sum_{t=0}^{T-1}\left(x_t^{\top} Q x_t+u_t^{\top} R u_t\right)+x_T^{\top} Q_T x_T\right],
    \end{equation}
where $\mathcal{B}$ is an ambiguity set is defined as
\begin{equation}\label{ambiguity:set}
\mathcal{B}= \mathcal{B}_{x_0} \otimes \left(\otimes_{t=0}^{T-1} \mathcal{B}_{w_t}\right) \otimes\left(\otimes_{t=0}^{T-1} \mathcal{B}_{v_t}\right)\\
\end{equation}
where
$$
\begin{aligned}
& \mathcal{B}_{x_0}\!=\!\left\{\mathbb{P}_{x_0} \! \in\! \mathcal{P}\left(\mathbb{R}^n\right): R\left({\mathbb{P}}_{x_0}|| \hat{\mathbb{P}}_{x_0}\right)\! \leq\! \rho_{x_0}, \: \mathbb{E}_{\mathbb{P}_{x_0}}[x_0] \!=\! 0\right\} \\
& \mathcal{B}_{w_t}\!=\!\left\{\mathbb{P}_{w_t} \!\in \!\mathcal{P}\left(\mathbb{R}^n\right): R\left({\mathbb{P}}_{w_t}|| \hat{\mathbb{P}}_{w_t}\right) \!\leq \!\rho_{w_t}, \: \mathbb{E}_{\mathbb{P}_{w_t}}[w_t] \!=\! 0\right\} \\
& \mathcal{B}_{v_t}\!=\!\left\{\mathbb{P}_{v_t} \!\in\! \mathcal{P}\left(\mathbb{R}^m\right): R\left({\mathbb{P}}_{v_t}|| \hat{\mathbb{P}}_{v_t}\right) \!\leq \!\rho_{v_t}, \: \mathbb{E}_{\mathbb{P}_{v_t}}[v_t] \!=\! 0\right\},
\end{aligned}
$$
for user-defined $\rho_{x_0}, \rho_{w_t}, \rho_{v_t} \geq 0$. Note that by construction all exogenous random variables $x_0, w_0, \ldots, w_{T-1}, v_0, \ldots, v_{T-1}$ are mutually independent under every distribution in $\mathcal{B}$. The restriction to zero-mean distributions, both in the nominal and in the perturbed distributions, is done to simplify the presentation; the results can be extended to the case of non-zero mean distributions with minor modifications. Note also that the DR-LQG problem \eqref{eq:DR-LQG} constitutes a zero-sum game between the control policy and a fictitious adversary that selects $\mathcal{B}$.

\subsection{On the choice of the KL-divergence}
We briefly review some benefits of the use of KL divergence to define ambiguity sets.


\textbf{Connection with system identification.} Consider the problem of estimating the parameters $\theta = (A,B,C)$ of a linear system of the form \eqref{eq:nominal:system} from noisy input-output data $\{(u_t, y_t)\}$. One of the most widely used paradigms to accomplish this task is maximum likelihood estimation (MLE), where the log-likelihood $\mathcal{L}(\theta)$ is maximized with respect to $\theta$. \cite{ljung1999system} provides an attractive interpretation of the MLE problem as the search for a model that minimizes the KL-divergence to the true system. Building on this interpretation, we can directly employ $\mathcal{L}(\theta^\star)$ to provide a meaningful estimate of the size of the KL ambiguity sets.


\textbf{Connection with risk measures.} For $p \in [1,\infty)$, let $\mathcal{L}_p(\Omega, \mathcal{F}, \mathbb{P})$ be the space of random variables $\xi: \Omega \rightarrow \mathbb{R}$ with finite $p-$th moment with respect to the measure $\mathbb{P}$. Then, a risk measure  $\rho(\xi)$ with $\rho: \mathcal{L}_p(\Omega, \mathcal{F}, \mathbb{P}) \rightarrow {\mathbb{R}} \cup \{\infty\}$ is a mapping that maps $\xi$ to the extended real line, quantifying its "riskiness". Originating in economics, risk measures are now popular in the control community to allow for a systematic approach to risk assessment (e.g., for safety-critical systems) \cite{akella2024risk}. Among them, the class of coherent risk measures is the most widespread thanks to their appealing properties \cite[Chapter~6, p.231]{shapiro2021lectures}. Exploiting the Fenchel-Moreau theorem, it is possible to show that \textit{any} coherent risk measure can be written as
\begin{equation}\label{eq:coherent}    
\mathfrak{R}(\xi) = \sup_{\frac{d\mathbb{Q}}{d\mathbb{P}}\in \mathcal{A}} \mathbb{E}_{\mathbb{Q}}[\xi], \quad \forall \xi \in \mathcal{L}_p(\Omega, \mathcal{F}, \mathbb{P}),
\end{equation}
where $\mathcal{A}$ is a set of probability density functions, and the measure $\mathbb{Q}$ is such that $\mathbb{Q} << \mathbb{P}$. This provides a clear  connection with distributionally robust optimization with KL-based ambiguity sets.

\begin{example}[Conditional Value at Risk]
The CVaR of level $\beta \in (0,1)$ can be described as in \eqref{eq:coherent} with \cite{shapiro2021lectures} 
$$
\mathcal{A} = \textstyle\left\{ \mathbb{Q} << \mathbb{P} \:\Big| \:\frac{d\mathbb{Q}}{d\mathbb{P}} \geq 0, \mathbb{E}_{\mathbb{P}}\left[\frac{d\mathbb{Q}}{d\mathbb{P}}\right] = 1, \frac{d\mathbb{Q}}{d\mathbb{P}} \leq \frac{1}{\beta} \right\}. 
$$
Let $p(x), q(x)$ be the densities of $\mathbb{P}, \mathbb{Q}$. We have:
$$
\begin{aligned}
\textstyle \int q(x) \log \left( \frac{q(x)}{p(x)} \right) dx \leq  \log\left(\frac{1}{\beta}\right) \int d\mathbb{Q}  =  \log\left(\frac{1}{\beta}\right),
\end{aligned}
$$
since $\int q(x) dx = 1$. Consider now the LQG functional in \eqref{eq:DR-LQG}, and assume that instead of the expectation we want to consider the CVaR as we are interested in accounting for the tail-risk. Then, the cost functional becomes
\begin{equation}\label{eq:CVAR:DR}
\begin{aligned}
& \inf_{\pi} \:\mathfrak{R}\left[\sum_{t=0}^{T-1}\left(x_t^{\top} Q x_t+u_t^{\top} R u_t\right)+x_T^{\top} Q_T x_T\right] \\
= &\min _{\mathbf{x}, \mathbf{u}} \: \mathfrak{R}\left[\mathbf{u}^{\top} \mathbf{R} \mathbf{u}+\mathbf{x}^{\top} \mathbf{Q} \mathbf{x}\right] 
\end{aligned}
\end{equation}
where for the second line we borrow the stacked notation from the upcoming Section \ref{sec:theory}, and the system evolves as in \eqref{eq:nominal:system}. For $\xi = \mathbf{u}^{\top} \mathbf{R} \mathbf{u}+\mathbf{x}^{\top} \mathbf{Q} \mathbf{x}$, we then conclude that $\sup_{\mathbb{Q} \in \mathcal{B}_\rho} \mathbb{E}_\mathbb{Q}[\xi]$, with $\rho = \textstyle \log\left(\textstyle\frac{1}{\beta}\right)$, represents a conservative approximation of \eqref{eq:CVAR:DR}. While solving \eqref{eq:CVAR:DR} is challenging \cite{chapman2021toward}, we will show that \eqref{eq:DR-LQG} can be solved effectively, thus paving the way for novel risk-aware formulations as \cite{tsiamis2021linear}. $\hfill \triangle$
\end{example}

\section{Theoretical analysis of the DR-LQG}\label{sec:theory}


\subsection{Problem re-parametrization}
As pointed out in \cite{taskesen2023distributionally}, in the LQG formulation the inputs are subject to a cyclic dependence, which makes the analysis of \eqref{eq:DR-LQG} hard. To break this dependency, we proceed as in \cite{taskesen2023distributionally} and introduce a "fictitious" noise-free system
\begin{equation}\label{eq:purified}
\begin{aligned}
\hat{x}_{t+1}&=A \hat{x}_t+B u_t, \quad  \hat{y}_t&=C \hat{x}_t 
\end{aligned}
\end{equation}
with states $\hat{x}_t \in \mathbb{R}^n$ and outputs $\hat{y}_t \in \mathbb{R}^p$, initialized with $\hat{x}_0=0$ and subject to the same inputs $u_t$ as the original system \eqref{eq:nominal:system}. We define the purified observation at time $t$ as $\eta_t=y_t-\hat{y}_t$ and use $\boldeta=\left(\eta_0, \ldots, \eta_{T-1}\right)$ to denote the sequence of purified observations.
Note that since the inputs $u_t$ are causal, we can compute $\hat{x}_t$ and $\hat{y}_t$ from $y_0, \ldots, y_t$. Thus, $\eta_t$ can be represented as a function of $y_0, \ldots, y_t$. Conversely, $y_t$ can also be represented as a function of $\eta_0, \ldots, \eta_t$. Moreover, any measurable function of $y_0, \ldots, y_t$ can be expressed as a measurable function of $\eta_0, \ldots, \eta_t$ and viceversa \cite[Proposition II.1]{hadjiyiannis2011efficient}. Let $\mathcal{U}_{\boldeta}$ be the sequence of control inputs $\bu = \left(u_0, u_1, \ldots, u_{T-1}\right)$ so that $u_t=\tilde{\pi}_t\left(\eta_0, \ldots, \eta_t\right)$ for some measurable function $\tilde{\pi}_t: \mathbb{R}^{p(t+1)} \rightarrow \mathbb{R}^m$ for every $t = \{0, \ldots, T-1\}$. Then, \cite[Proposition II.1]{hadjiyiannis2011efficient} implies that $\mathcal{U}_{\boldeta}=\mathcal{U}_\by$. 

Let $\mathbf{Q} \in \mathbb{S}^{n(T+1)}, \mathbf{R} \in \mathbb{S}^{m T}, \mathbf{C} \in \mathbb{R}^{p T \times n(T+1)}, \mathbf{G} \in \mathbb{R}^{n(T+1) \times n(T+1)}$ and $\mathbf{H} \in \mathbb{R}^{n(T+1) \times m T}$ be block matrices as defined in the Appendix A. Further, let $\mathbf{x} = (x_0,\ldots, x_T), \: \mathbf{u} = (u_0,\ldots, u_{T-1}), \: \mathbf{y} = (y_0,\ldots, y_{T-1}), \: \mathbf{w} = (x_0, w_0, \ldots, w_{T-1}), \: \mathbf{v} = (v_0,\ldots, v_{T-1})$. Using the stacked system matrices, we can now express the purified observation process $\mathbf{\eta}$ as a linear function of the exogenous uncertainties $w$ and $v$. Specifically, one can easily verify that $\boldsymbol{\eta}=\mathbf{D} \mathbf{w}+\mathbf{v}$, where $\mathbf{D}=\mathbf{C} \mathbf{G}$. Note that the purified observations are independent of the inputs, hence breaking the initial cyclic dependence. In view of this re-parametrization, \eqref{eq:DR-LQG} is equivalent to
\begin{equation}
 \begin{aligned}
(P) := \quad & \min _{\mathbf{x}, \mathbf{u}}  \max _{\mathbb{P} \in \mathcal{B}} \mathbb{E}_{\mathbb{P}}\left[\mathbf{u}^{\top} \mathbf{R} \mathbf{u}+\mathbf{x}^{\top} \mathbf{Q} \mathbf{x}\right] \\
& \text { s.t. } \:  \bu \in \mathcal{U}_{\boldeta}, \quad \bx=\bH \bu+\bG \bw.
\end{aligned}
\end{equation}
In the remainder, we will then focus without loss of generality on the re-parametrized problem $(P)$. 


\subsection{Analysis of the upper bound}
Consider the following problem\begin{equation}\label{eq:upper}
    \begin{aligned}
   (U) := \quad & \min _{\bx, \bu}  \max _{\mathbb{P} \in \tilde{\mathcal{B}}} \mathbb{E}_{\mathbb{P}}\left[\bu^{\top} \bR \bu+\bx^{\top} \bQ \bx\right] \\
   & \text { s.t. } \:  \bu \in \mathcal{U}^{\text{lin}}_{\boldeta}, \quad \bx=\bH \bu+\bG \bw,
   \end{aligned}
   \end{equation}
   where $\tilde{\mathcal{B}}$ is an ambiguity set (to be defined next) such that $\mathcal{B} \subseteq \tilde{\mathcal{B}}$, and $\mathcal{U}^{\text{lin}}_{\boldeta} $ denotes the class of affine policies of the form $\bu = \bU \boldeta + \bq$ with $\bU \in \mathbb{R}^{m T \times p T}$ being block lower triangular to enforce causality and $\bq \in \mathbb{R}^{m T}$. Clearly, $(U)$ is an upper bound to $(P)$ since we are simultaneously restricting the feasible space of the controller, while enlarging the one of the adversary. 
   We begin by defining $\tilde{\mathcal{B}}$.
\begin{proposition}\label{KL:bound}
    Consider an arbitrary zero-mean distribution $\mathbb{P} \in \mathcal{P}(\mathbb{R}^n)$ with covariance $\Sigma_p$ and a Gaussian distribution $\mathbb{Q} \triangleq \mathcal{N}(0, \Sigma_q) \in \mathcal{P}_{\mathcal{G}}(\mathbb{R}^n)$. Then,
    \begin{equation}
    R(\mathbb{P} \| \mathbb{Q}) \geq \frac{1}{2} \left( \emph{Tr}(\Sigma_q^{-1} \Sigma_p) - n + \log \frac{\det \Sigma_q}{\det \Sigma_p} \right).
    \end{equation}
    \end{proposition}
 The bound in Proposition \ref{KL:bound} depends only on the covariance matrices $\Sigma_p$ and $\Sigma_q$, making it analogous to the Gelbrich bound for the type-2 Wasserstein distance. Further, the derivation can be easily extended to the case of non-zero mean distributions, since the differential entropy is translation invariant. Consider the setting of Proposition \ref{KL:bound} and let  ${\mathcal{B}} = \{\mathbb{P}\; : \; R(\mathbb{P} \| \mathbb{Q})  \leq \rho \}$ and $\tilde{\mathcal{B}} = \{\mathbb{P}\; : \; \frac{1}{2} \left( \text{Tr}(\Sigma_q^{-1} \Sigma_p) - n + \log \frac{\det \Sigma_q}{\det \Sigma_p} \right) \leq \rho \}$. It immediately follows
 \begin{equation}\label{lemma:ball:inclusion}
 \mathcal{B} \subseteq \tilde{\mathcal{B}}.
    \end{equation}
   
    Next, we turn our attention to the minimization problem. By the definition of $\boldeta$ and the restriction to linear policies,
    \begin{equation}\label{eq:linear:policy}
     \bu = \bU(\bD\bw + \bv) + \bq,\quad \bx = \bH\left( \bU(\bD\bw + \bv) + \bq \right) + \bG\bw.
     \end{equation}
     Taking the expectation of the quadratic form and by definition of $\tilde{\mathcal{B}}$, the upper bound $(U)$ becomes\footnote{To simplify the notation, we shift the index $i$ of $W_i$ by one.}
\begin{equation*}\label{eq:reformulation:1}
\begin{aligned}
\min_{\bU,\bq} \max_{\bW,\bV \succeq 0} \: & \text{Tr}(\bF_1 \bW +  \bF_2 \bV) + \bq^\top \bM \bq\\
\text{s.t.} \: & \bW = \text{diag}(W_0, \ldots, W_{T}), \\
 & \bV = \text{diag}(V_0, \ldots, V_{T-1}),\\
& \frac{1}{2} \left( \text{Tr}(\hat{W}_t^{-1} W_t) \!-\! n \!+\! \log \frac{\det \hat{W}_t}{\det W_t} \right) \!\leq\! \rho_{w_t},  \forall t\\
& \frac{1}{2} \left( \text{Tr}(\hat{V}_t^{-1} V_t)\! -\! n \!+ \!\log \frac{\det \hat{V}_t}{\det V_t} \right)\! \leq \!\rho_{v_t},  \forall t.
\end{aligned}
\end{equation*}
where 
$$
\begin{aligned}
\bF_1 & \!= \!\bD^\top \bU^\top \bR \bU\bD \!+ \!(\bH\bU\bD \!+ \!\bG)^\top \bQ(\bH\bU\bD \!+\! \bG)\! \in \!\mathbb{S}_{+}^{nT},\\
\bF_2 & \!=\! \bU^\top \bR \bU \!+\! \bU^\top  \bH^\top \bQ \bH \bU \!\in\! \mathbb{S}_{+}^{pT},\\
\bM & \!=\! \bR \!+\! \bH^\top \bQ \bH \!\in \!\mathbb{R}^{nT}.
\end{aligned}
$$

\subsection{Analysis of the lower bound}
Consider now the problem
\begin{equation}\label{eq:lower:2}
 \begin{aligned}
(L) := \quad & \min _{\bx, \bu}  \max _{\mathbb{P} \in \mathcal{G}} \mathbb{E}_{\mathbb{P}}\left[\bu^{\top} \bR \bu+\bx^{\top} \bQ \bx\right] \\
& \text { s.t. } \:  \bu \in \mathcal{U}_{\boldeta}, \quad \bx=\bH \bu+\bG \bw,
\end{aligned}
\end{equation}
where $\mathcal{G}$ is defined as 
$$\mathcal{G}= \mathcal{G}_{x_0} \otimes \left(\otimes_{t=0}^{T-1} \mathcal{G}_{w_t}\right) \otimes\left(\otimes_{t=0}^{T-1} \mathcal{G}_{v_t}\right)$$ 
with 
$$
\begin{aligned}
& \mathcal{G}_{x_0}\!=\!\left\{\mathbb{P}_{x_0} \!\in\! \mathcal{P}_{\mathcal{G}}\left(\mathbb{R}^n\right): R\left({\mathbb{P}}_{x_0}, \hat{\mathbb{P}}_{x_0}\right) \!\leq \!\rho_{x_0}, \: \mathbb{E}_{\mathbb{P}_{x_0}}[x_0] \!=\! 0\right\} \\
& \mathcal{G}_{w_t}\!=\!\left\{\mathbb{P}_{w_t} \!\in \!\mathcal{P}_{\mathcal{G}}\left(\mathbb{R}^n\right): R\left({\mathbb{P}}_{w_t}, \hat{\mathbb{P}}_{w_t}\right)\! \leq \!\rho_{w_t}, \: \mathbb{E}_{\mathbb{P}_{w_t}}[w_t] \!=\! 0\right\} \\
& \mathcal{G}_{v_t}\!=\!\left\{\mathbb{P}_{v_t} \!\in \!\mathcal{P}_{\mathcal{G}}\left(\mathbb{R}^m\right): R\left({\mathbb{P}}_{v_t}, \hat{\mathbb{P}}_{v_t}\right) \!\leq \!\rho_{v_t}, \: \mathbb{E}_{\mathbb{P}_{v_t}}[v_t] \!=\! 0\right\}.
\end{aligned}
$$
Note that $\mathcal{G}$ only contains Gaussian distributions, hence $\mathcal{G} \subseteq \mathcal{B}$ and $(L)$ is a lower bound for $(P)$. Moreover, notice that for any $\mathbb{P} \in \mathcal{G}$, from classical LQG theory, the optimal policy is affine in the observations so we can rewrite $(L)$ as
\begin{equation}\label{eq:lower:2}
    \begin{aligned}
   (L) := \quad & \min _{\bx, \bu}  \max _{\mathbb{P} \in \mathcal{G}} \mathbb{E}_{\mathbb{P}}\left[\bu^{\top} \bR \bu+\bx^{\top} \bQ \bx\right] \\
   & \text { s.t. } \:  \bu \in \mathcal{U}_{\boldeta}^{\text{lin}}, \quad \bx=\bH \bu+\bG \bw.
   \end{aligned}
   \end{equation}

Invoking \eqref{eq:tightness}, the lower bound $(L)$ becomes
\begin{equation*}\label{eq:lower:reformulation}
    \begin{aligned}
    \min_{\bU, \bq} \max_{\bW,\bV \succeq 0} \: & \text{Tr}(\bF_1 \bW +  \bF_2 \bV) + \bq^\top \bM \bq\\
    \text{s.t.} \: & \bW = \text{diag}(W_0, \ldots, W_{T}), \\
     & \bV = \text{diag}(V_0, \ldots, V_{T-1}),\\
    & \frac{1}{2} \left( \text{Tr}(\hat{W}_t^{-1} W_t) \!-\! n \!+\! \log \frac{\det \hat{W}_t}{\det W_t} \right) \!\leq \!\rho_{w_t}, \forall t\\
    & \frac{1}{2} \left( \text{Tr}(\hat{V}_t^{-1} V_t) \!- \!n \!+ \!\log \frac{\det \hat{V}_t}{\det V_t} \right)\! \leq \!\rho_{v_t},  \forall t.
    \end{aligned}
    \end{equation*} 

\subsection{Optimality of linear policies}
The analysis of the upper and lower bounds culminates with the following result, which represents our first contribution.
\begin{theorem}\label{thr:lin:gauss}
    Problem (P) is solved by an affine policy $\bu^\star = \bU^\star \by + \bq^\star$ and a Gaussian distribution $\mathbb{P}^\star \in \mathcal{G}$.
    \end{theorem}

Theorem \ref{thr:lin:gauss} establishes that the optimal policy for the DR-LQG problem is linear in the observations, as in the standard LQG setting. Moreover, the worst-case distribution is still a Gaussian distribution even though the ambiguity set $\mathcal{B}$ contains different types of distributions.

    \section{Computational framework for the DR-LQG}\label{Sec:computation}
    In search of methods for computing the optimal $u^\star$ and $\mathbb{P}^\star$, we start by showing that problem $(L)$ admits a Nash equilibrium. Then we devise a best response algorithm that provably converges to te set of saddle points of \eqref{eq:lower:2}. Note that from the structure of (U) and (L), we directly infer that $\mathbf{q}^\star = \mathbf{0}$ when the nominal distribution is zero-mean. To simplify the results presentation, we will drop $\mathbf{q}$ in the remainder. All our results can be easily extended to the case of non-zero mean distributions, e.g., $\mathbf{q}^\star \neq \mathbf{0}$ with minor modifications.

    \subsection{Existence of a Nash equilibrium}
    Consider problem \eqref{eq:lower:2} and its dual 
    \begin{equation}\label{eq:dual}
        \begin{aligned}
       (D) := \quad &  \max _{\mathbb{P} \in \mathcal{G}} \min _{\bx, \bu} \mathbb{E}_{\mathbb{P}}\left[\bu^{\top} \bR \bu+\bx^{\top} \bQ \bx\right] \\
       & \text { s.t. } \:  \bu \in \mathcal{U}_{\boldeta}^{\text{lin}}, \quad \bx=\bH \bu+\bG \bw.
       \end{aligned}
       \end{equation}
    
       The next result shows that strong duality among the two problems holds.

       \begin{proposition}\label{strong:duality}
       Problem \eqref{eq:dual} is a strong dual for \eqref{eq:lower:2}.
       \end{proposition}

As a consequence of Proposition \ref{strong:duality}, \eqref{eq:linear:policy} 
admits a saddle point $(\bu^\star, \mathbb{P}^\star)$. 


\subsection{Regularized best response scheme}

Let $\bSigma = \text{diag}(\bW, \bV)$ and $\bF(\bU) = \text{diag}(\bF_1(\bU), \bF_2(\bU))$, where we stress the dependence of $F$ from $\bU$. Further, let $\mathcal{J}(\bU,\bSigma) := \text{Tr}\left( \bF(\bU) \bSigma\right)$ be the "payoff" function of the zero-sum game described by \eqref{eq:lower:2} and set
\begin{equation}
    \begin{aligned}
\mathcal{A}(\bU) & = \max_{\Sigma \in \mathcal{G}} \mathcal{J}(\bU,\bSigma),\\
\mathcal{C}(\bSigma) & = \min_{\bU} \mathcal{J}(\bU,\bSigma).
    \end{aligned}
\end{equation}
Since $\mathcal{J}$ is convex in $\bU$ and concave in $\bSigma$, by Berge's Maximum theorem $\mathcal{A}$ is convex and continuous in $\bU$ and $\mathcal{C}$ is concave and continuous in $\bSigma$. Additionally, one has $\mathcal{C}(\bSigma) \leq \mathcal{J}(\bU,\bSigma) \leq \mathcal{A}(\bU)$, hence
$$
\begin{aligned}
\underline{\mathcal{V}} := & \max_{\bSigma \in \mathcal{G}} \mathcal{C}(\bSigma) = \max_{\bSigma \in \mathcal{G}} \min_{\bU} \mathcal{J}(\bU,\bSigma)  \\
& \leq \min_{\bU} \max_{\bSigma \in \mathcal{G}} \mathcal{J}(\bU) = \min_{\bU} \mathcal{A}(\bU) =: \overline{\mathcal{V}}.
\end{aligned}
$$
Let $$\Phi(\bU,\bSigma) := \mathcal{A}(\bU) -\mathcal{C}(\bSigma) $$
be the residual duality gap. Then: $\Phi: \mathbb{R}^{mT \times pT} \times \mathcal{G}\rightarrow \mathbb{R}$ is continuous, jointly convex and nonnegative, and its minimum $\overline{\mathcal{V}} - \underline{\mathcal{V}} \geq 0$ is reached in $\mathbb{R}^{mT \times pT} \times \mathcal{G}$. If the minimum is 0, then the zero-sum game is said to have a value $\mathcal{V} = \overline{\mathcal{V}} = \underline{\mathcal{V}}$.  More precisely, $\Phi(\bU,\bSigma)  = 0$ if and only if $\mathcal{A}(\bU) = \mathcal{C}(\bSigma)$, or equivalently, $(\bU,\bSigma)$ is a saddle point of the game.

Let us introduce the best response (BR) correspondences
\begin{subequations}\label{eq:BR}
    \begin{align}\label{BR1}
        \text{BR}_1(\bSigma) & \in \text{argmin}_{\bU} \:  \mathcal{J}(\bU,\bSigma)\\
        \text{BR}_2(\bU) & \in \text{argmax}_{\bSigma \in \mathcal{G}} \: \mathcal{J}(\bU,\bSigma).
    \end{align}
\end{subequations}

Before proceeding, we show that \eqref{BR1} is equivalently obtained by restricting $\bU$ to live in a compact set ${\mathcal{S}}\subseteq \mathbb{R}^{mT \times pT}$.
\begin{lemma}\label{set:S:compact}
There exists a compact and convex set ${\mathcal{S}}\subseteq \mathbb{R}^{mT \times pT}$ independent of $\bSigma$ such that $\text{BR}_1(\bSigma) = \text{argmin}_{\bU \in {\mathcal{S}}} \:  \mathcal{J}(\bU,\bSigma)$ for all $\bSigma \in \mathcal{G}$.
\end{lemma}



Since $\mathcal{J}(\bU,\bSigma)$ is continuous in both arguments, Maximum theorem implies that $\text{BR}_1, \text{BR}_2$ are upper semi-continuous correspondences from $\mathcal{G}$ to $\mathcal{S}$ and from $\mathcal{S}$ to $\mathcal{G}$ respectively, with nonempty closed convex values.
Consider the continuous-time best response dynamics on $\mathcal{S} \times \mathcal{G}$:
\begin{subequations}\label{eq:BR:continuous}
\begin{align}
\dot{\bU}(\lambda) & \in \text{BR}_2(\bSigma(\lambda)) - \bU(\lambda)\\
\dot{\bSigma}(\lambda) & \in \text{BR}_1(\bU(\lambda)) - \bSigma(\lambda).
\end{align}
\end{subequations}
We can interpret \eqref{eq:BR:continuous} as a regularized best response scheme, where the regularization, given by the negative of the best response of the corresponding player, is added to dampen oscillations and promote convergence. Denote $v(\lambda) = \Phi(\bU(\lambda),\bSigma(\lambda))$.  We have the following result.

\begin{theorem}\label{thr:convergence:CT} The following facts holds:
\begin{enumerate}
\item [i.] \eqref{eq:BR:continuous} has a solution for every initial condition $(\bU(0),\bSigma(0)) \in \mathcal{S} \times \mathcal{G}$.
\item [ii.] Along every solution of \eqref{eq:BR:continuous}, $v(\lambda)$ is a Lyapunov function with
$
 \dot{v}(\lambda) \leq - v(\lambda) \quad \text{for almost all } \lambda,
$
hence
\begin{equation}
v(\lambda) \leq  e^{-\lambda} v(0).
\end{equation}
Thus, the zero-game has a value and every solution of \eqref{eq:BR:continuous} converges to the non-empty set of saddle points, which is a uniform global attractor for \eqref{eq:BR:continuous}.
\end{enumerate}
\end{theorem}

Theorem \ref{thr:convergence:CT} establishes that the (regularized) best response scheme in \eqref{eq:BR:continuous} converges at an exponential rate to the set of saddle points of the DR-LQG problem. 

Moreover, we notice that both best responses admits a computationally cheap solution, making our scheme attractive from a computational point of view. Namely, the best response of the minimizing player can be solved via a closed-form dynamic programming recursion coupled with a Kalman filter, as in standard LQG. In turn, we show the best response of the maximizing player can be decomposed as a sum of $2T + 1$ independent convex optimization problems in a single scalar variable $\tau_i \in \mathbb{R}_+$ each of the form
\begin{equation}\label{eq:max:subproblem}
    \begin{aligned}
\bSigma_i = \:  & \text{argmax}_{\bSigma_i \succ 0} \quad\text{Tr}\left(\bF_{i}(\bU) \bSigma_i\right) \\
& \text{s.t.} \quad \frac{1}{2}  \left( \text{Tr}(\hat{\bSigma}_i^{-1} \bSigma_i) - n + \log \frac{\det \hat{\bSigma}_i}{\det \bSigma_i} \right) \leq \rho_i,
    \end{aligned}
\end{equation}
where $\bF_{i}$ is the $i$-th digoanl block of $\bF(\bU)$ and $\hat{\bSigma}_i$ the $i$-th diagonal block of $\hat{\bSigma} = \text{diag}(\hat{W}_{0}, \ldots, \hat{W}_T, \hat{V}_0, \ldots, \hat{V}_{T-1})$.

We show that \eqref{eq:max:subproblem} admits a closed-form solution. Exploiting Lemma \ref{lemma:opt:variance} since each $\bF_{i}(\bU) \in \mathbb{S}_{+}^n$ for any $\bU$, the maximizer of \eqref{eq:max:subproblem} is given by $$\bSigma_i^\star(\tau_i^\star) =\textstyle\left(\hat{\bSigma}_i^{-1} - \frac{2 \bF_{i}(\bU)}{\tau_i^\star} \right)^{-1},$$
where by Lagrange duality $\tau_i^\star$ is chosen such that complementarity slackness holds, i.e., \begin{equation}\label{bisection}
\frac{1}{2} \left( \text{Tr}(\hat{\bSigma}_i^{-1} \bSigma_i^\star(\tau_i)) - n + \log \frac{\det \hat{\bSigma}_i}{\det \bSigma_i^\star(\tau_i)} \right) = \rho_i.
\end{equation}
 Note that \eqref{bisection} represents an algebraic equation in the scalar variable $\tau_i \in \mathbb{R}_+$. In practice, it can be efficiently solved by bisection to any desired tolerance.

    \section{DR-LQG with endogenous ambiguity sets}\label{sec:endo}
    The problem formulation considered up to here assumes exogenous disturbances resulting in ambiguity sets that are defined a priori, i.e., before the control task, and are not influenced by the controller's actions. In this section, we extend the results to the case where the ambiguity sets are endogenously determined by the controller's actions.

    Consider the standard Linear Fractional Transformation (LFT) model in Fig.~\ref{fig:LFT}: the source of uncertainty in the system dynamics, represented by the uncertainty operator $\Delta$ that is assumed to be unknown but of bounded magnitude, might depend on the system states and inputs. Consequently, to adequately represent the distributional ambiguity affecting the system, we shall resort to a state and decision-dependent ambiguity set of the form
    \begin{equation} \label{eq:relative:entropy:constraint_gauss:extension}
      {\luci \mathcal{B}_t(x_t, u_t) }= \{ \mathbb{P} \in \mathcal{P}_\mathcal{G}(\mathbb{R}^n) \: : \: R(\mathbb{P}|| \mathbb{Q}_t) \leq \rho_t + \frac{1}{2} \Vert z_t \Vert^2 \}.
    \end{equation}
    for a nominal Gaussian distribution $\mathbb{Q}_t$, with $ z_t = E_1x_t + E_2 u_t$ where $E_1 \in \R^{q \times n}$ and $E_2 \in \R^{q \times m}.$ For simplicity, we assume that $E_1^\top E_2 = 0.$ 
    \begin{figure} 
    \centering
            \tikzset{every picture/.style={line width=0.75pt}} 
        \begin{tikzpicture}[x=0.75pt,y=0.75pt,yscale=-0.8,xscale=0.8] 
            
            \draw   (293.72,51) -- (371.39,51) -- (371.39,105.21) -- (293.72,105.21) -- cycle ;
            \draw   (293.72,160.79) -- (371.39,160.79) -- (371.39,215) -- (293.72,215) -- cycle ;
            \draw    (167,201) -- (279,200.88) -- (291,200.98) ;
            \draw [shift={(294,201)}, rotate = 180.44] [fill={rgb, 255:red, 0; green, 0; blue, 0 }  ][line width=0.08]  [draw opacity=0] (8.93,-4.29) -- (0,0) -- (8.93,4.29) -- cycle    ;
            \draw    (371.39,200.09) -- (467,200) ;
            \draw [shift={(470,200)}, rotate = 179.95] [fill={rgb, 255:red, 0; green, 0; blue, 0 }  ][line width=0.08]  [draw opacity=0] (8.93,-4.29) -- (0,0) -- (8.93,4.29) -- cycle    ;
            \draw    (371.39,174.34) -- (436.86,174.34) ;
            \draw    (436.64,87.95) -- (436.75,174.69) ;
            \draw    (436.64,87.95) -- (374.39,88.9) ;
            \draw [shift={(371.39,88.95)}, rotate = 359.12] [fill={rgb, 255:red, 0; green, 0; blue, 0 }  ][line width=0.08]  [draw opacity=0] (8.93,-4.29) -- (0,0) -- (8.93,4.29) -- cycle    ;
            \draw   (216.04,176.37) .. controls (216.04,168.51) and (221.51,162.14) .. (228.25,162.14) .. controls (234.99,162.14) and (240.46,168.51) .. (240.46,176.37) .. controls (240.46,184.23) and (234.99,190.6) .. (228.25,190.6) .. controls (221.51,190.6) and (216.04,184.23) .. (216.04,176.37) -- cycle ;
            \draw    (165,176.37) -- (213.04,176.37) ;
            \draw [shift={(216.04,176.37)}, rotate = 180] [fill={rgb, 255:red, 0; green, 0; blue, 0 }  ][line width=0.08]  [draw opacity=0] (8.93,-4.29) -- (0,0) -- (8.93,4.29) -- cycle    ;
            \draw    (293.61,90.31) -- (228.25,90.31) ;
            \draw    (228.25,90.31) -- (228.25,159.14) ;
            \draw [shift={(228.25,162.14)}, rotate = 270] [fill={rgb, 255:red, 0; green, 0; blue, 0 }  ][line width=0.08]  [draw opacity=0] (8.93,-4.29) -- (0,0) -- (8.93,4.29) -- cycle    ;
            \draw    (240.46,176.37) -- (290.72,177.01) ;
            \draw [shift={(293.72,177.05)}, rotate = 180.73] [fill={rgb, 255:red, 0; green, 0; blue, 0 }  ][line width=0.08]  [draw opacity=0] (8.93,-4.29) -- (0,0) -- (8.93,4.29) -- cycle    ;
            
            \draw (302 ,172) node [anchor=north west][inner sep=0.75pt]   [align=center] {Nominal\\System};
            \draw (320 , 70) node [anchor=north west][inner sep=0.75pt]  [font=\Large]  {$\Delta $};
            \draw (255,  205) node [anchor=north west][inner sep=0.75pt]    {$u_{t}$};
            \draw (438,  205) node [anchor=north west][inner sep=0.75pt]    {$x_{t}$};
            \draw (400, 73) node [anchor=north west][inner sep=0.75pt]    {$ z_{t}$};
            \draw (237, 150) node [anchor=north west][inner sep=0.75pt]    [align=center] {${w}_{t} \!\sim \! \mathbb{P}_t$};
            \draw (162, 144) node [anchor=north west][inner sep=0.75pt]    [align=center] {$w_{t} \!\sim \!\hat{\mathbb{P}}_t$};		
        \end{tikzpicture}
         \caption{LFT uncertain model. It separates the nominal dynamics from the uncertainty, represented by the operator $\Delta$ in a feedback interconnection. By suitably selecting the admissible structure and magnitude of $\Delta$, the LFT model can represent various sources of uncertainty.} \label{fig:LFT}
    \end{figure}
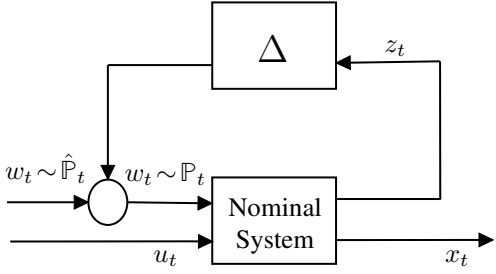
    
    While the theoretical analysis of Section IV carries over, the convergence of the best response dynamics in \eqref{eq:BR:continuous} is no longer guaranteed, since the endogenous disturbance introduces a coupling among the feasible sets of the players. Thus, Theorem \ref{thr:convergence:CT} does no longer necessarily hold. 
    
    To overcome this issue, we devise a different approximated scheme that retain convergence guarantees despite the additional complexity introduced by the endogenous ambiguity sets. Differently from the previous sections, this approximated scheme is grounded on standard tools from dynamic programming. In this sense, the approach used here is therefore  more closely related to \cite{petersen2000minimax,hakobyan2022wasserstein,hakobyan2024wasserstein,kim2023distributional,falconi2025distributionally}.

    \subsection{Relaxed problem}\label{sec:approx}
    To simplify the derivation, in this section we only consider distributional ambiguity on the process noise $w_t$, while $\mathbb{P}_{x_0}, \mathbb{P}_{v_t}$ are assumed to be known and equal to the their nominal distributions, i.e., we set $\rho_{x_0} = \rho_{v_t} = 0, \: \forall \: t$. Let us denote the information collected up to time $t$ as 
    $$I_t = (y_0, \ldots, y_{t}, u_0, \ldots, u_{t-1}),$$
    with $I_0 = y_0$, and we call it the information vector. When a new control action is computed, the information vector is updated as
    $
    I_{t+1} = (I_t, y_{t+1}, u_t).
    $ It is well-known in stochastic optimal control theory that $I_t$ serves as sufficient statistics. Thus, the policies of the two players can be writtem as $u_t = \pi_t(I_t)$ and, similarly, $\mathbb{P}_{w_t} = \gamma_t(I_t),$ where $\gamma_t$ is a measurable function that maps the information vector $I_t$ to a distribution $\mathbb{P}_{w_t}$ in $\mathcal{B}_t(x_t, u_t)$.

    Let 
    $ {\Sigma}_t \!:=\! 
    \mathbb{E}_{\mathbb{P}_{w_t}} \left[
    \left(x_t \!-\! \mathbb{E}_{\mathbb{P}_{w_t}} \![x_t | I_t ] 
    \right)^\top\!
    (x_t \!-\! \mathbb{E}_{\mathbb{P}_{w_t}} \![x_t | I_t ] ) \; \big | \;  I_t 
    \right] 
    $
    be the prediction error covariance matrix of the state $x_t$ assuming that  system \eqref{eq:nominal:system} evolves according to the distribution $\mathbb{P}_{w_t}$ selected by the adversary. Similarly, we denote $\hat{\Sigma}_t$ the corresponding nominal quantity, i.e., computed with respect to the nominal distribution $\hat{\mathbb{P}}_{w_t}$. Consequently the DR-LQG problem with endogenous ambiguity sets can be formulated as
    \begin{equation}\label{eq:problem:DP}
        \inf_{\pi} \sup_{\gamma \in \Gamma}\textstyle \mathbb{E}_{\mathbb{P}_{w}} \left[\sum_{t=0}^{T-1}\left(x_t^{\top} Q x_t+u_t^{\top} R u_t\right)+x_T^{\top} Q_T x_T\right],
    \end{equation}
    where 
$ \Gamma \coloneqq \{\gamma = (\gamma_0, \ldots, \gamma_N) \: | \: \gamma_t : \R^{(p+m)t} \to \mathcal{P}_{\mathcal{G}}(\R^n),  \gamma_t (I_t)  \in {\mathcal{B}_t(x_t,u_t)} \}$. 

We begin by focusing on a relaxed version of Problem \eqref{eq:problem:DP} where, instead of constraining the adversary to select a distribution from the ambiguity sets $\luci \{ \mathcal{B}_t(x_t, u_t) \}_{t =0}^{T-1} $, we simply penalize the deviation of the distributions $\mathbb{P}_{w_t}$ from their nominal ones. In other words, we focus on the \textit{relaxed regularized} problem. Let us use for brevity $R(\mathbb{P}|| \mathbb{Q}_t) \equiv \mathcal{R}_t$. The relaxed problem reads

    \begin{equation}\label{eq:problem:relaxed}
    \inf_{\pi}\sup_{\gamma \in {\Gamma}} J^\tau(\pi, \gamma)
    \end{equation}
    where $J^\tau(\pi, \gamma) $ is defined as
    $$
    \begin{aligned}
     {\mathbb{E}_{\mathbb{P}_w}}\textstyle\left[ \sum_{t=0}^{T-1} \left(\frac{1}{2} (x_t^\top Q x_t + u_t^\top R u_t) - \tau_t \mathcal{R}_t \right) + \frac{1}{2}x_{T}^\top Q_{T} x_{T}\right],
    \end{aligned}
    $$
    for fixed $\{\tau_t\}_{t = 0}^N$ with $\tau_t \in \mathbb{R}_+$. 
    We address \eqref{eq:problem:relaxed}
    by  resorting to the dynamic programming technique. To this end, we derive the following instrumental result.
    
    
    \begin{proposition}\label{propo:recursion}
    Let  
    $\tilde{x}_t := {\mathbb{E}_{\mathbb{P}_{w_t}}}[x_t|I_{t}], $ 
    $ \xi_{t} := x_{t} - \tilde{x}_t $  and 
    ${\Sigma}_t := \mathbb{E}_{\mathbb{P}_{w_t}}[\xi_{t} \xi_{t}^\top|I_{t}]$  for $t \in \mathbb{N}.$ 
    Given $P_{t+1} \in \mathbb{S}^n_{++}$, $S_{t+1} \in \mathbb{S}_n$, $z_{t+1} \in \R$, define the function 
    $ r_{t} : \mathbb{S}_{++}^n \to \mathbb{R} $ as
    \begin{equation} \label{eq:r_rec}
    r_t (X) \! =  \! \frac{1}{2}\emph{Tr} (S_{t+1}
    (A X A^\top  \!-\! A X C^\top (C X C^\top \! + \!  \hat{V}_t)^{-1}C X A ^\top )) . 
    \end{equation} 
    Moreover, let $H_{t+1}(I_{t+1})$ be defined as
    $$
    \mathbb{E}_{\mathbb{P}_{w_t}}\textstyle \!\left[\frac{1}{2}(x_{t+1}^\top P_{t+1}x_{t+1} + \xi_{t+1}^\top S_{t+1}\xi_{t+1}) | I_{t+1} \right] \!+\! z_{t+1}. $$
    Consider the optimization problem
    \begin{equation}\label{eq:opt:value:function}
    \begin{aligned}
    H_t(I_{t}) = \inf_{u_t} \sup_{{\mu}_t, {W}_t\succ 0} &
  {\mathbb{E}_{\mathbb{P}_{w_t}}}\left[ \frac{1}{2}(x_{t}^\top Q x_t + u_t^\top R u_t) + \right. \\
  & \left. H_{t+1}(I_{t+1}) - \tau_t \mathcal{R}_t | I_{t}, u_t \right]
    \end{aligned}
    \end{equation}
    where the penalty parameter $\tau_t \geq 0$ is fixed and satisfies 
    \begin{align}
        \hat{W}_t^{-1 } - \frac{P_{t+1}}{\tau_t}\succ 0, \quad \hat{W}_t^{-1 } - \frac{P_{t+1} + S_{t+1}}{\tau_t}\succ 0.
        \label{eq:tau:conditionPS} 
    \end{align}
    Then:
    \begin{enumerate}
    \item[i. ] The optimal solutions of \eqref{eq:opt:value:function} are $ u_t^o = K^o_t \tilde{x}_t$ with 
    \begin{equation} \label{eq:drc_policy}
         K^o_t = -R^{-1}B^\top \textstyle \left(P_{t+1}^{-1} + BR^{-1}B - \frac{\hat{W}_t}{\tau_t}\right)^{-1}A,
    \end{equation}
     \begin{align}
    {\mu}_t^o & =\textstyle \left[ \left(I - \frac{P_{t+1}\hat{W}_t}{\tau_t}\right)^{-1} - I\right](A\tilde{x}_t + Bu_t^o), \label{eq:wc_mu}\\
    {W}_t^o & = \textstyle\left( \hat{W}_t^{-1} - \frac{P_{t+1} + S_{t+1}}{\tau_t} \right)^{-1}. \label{eq:wc_cov}
    \end{align}
    \item[ii. ]  The function $H_t(I_{t})$ in \eqref{eq:opt:value:function} is given by 
    \begin{align*}
        H_t(I_{t}) = {\mathbb{E}}_{\mathbb{P}_{x_{t}}}\textstyle\left[\frac{1}{2}(x_{t}^\top P_{t}x_{t} \!+\! \xi_{t}^\top S_{t}\xi_{t}) | I_{t}\right] \!+\! z_{t} \!+ \!r_t (\Sigma_t)
    \end{align*}
    where
    \begin{align}
    P_t & = Q + A^\top \textstyle\left( P_{t+1}^{-1} + BR^{-1}B^\top - \frac{\hat{W}_t}{\tau_t} \right)^{-1}A, 
    \label{eq:P_rec} \\
    S_t & = Q + A^\top P_{t+1}A - P_t, 
    \label{eq:S_rec} \\
    z_{t} &= z_{t+1} - \frac{\tau_t}{2}\textstyle \log \left| I - \frac{\hat{W}_t(P_{t+1} + S_{t+1})}{\tau_t}\right| 
    \label{eq:z_rec}
    \end{align}
    \end{enumerate}
    \end{proposition}

    Our aim is to apply Proposition \ref{propo:recursion} to derive an iterative scheme to solve \eqref{eq:problem:DP}. To this end, let
\begin{equation}\label{eq:V_exact_t}
\begin{aligned}
\mathcal{V}_{t} (I_t) \textstyle
:= 
\underset{\pi}{\inf} \:
\underset{\gamma \in {\Gamma}}{\sup} \:
\mathbb{E}_{\mathbb{P}_{w}}
& \left[
\sum_{s=t}^N \left(\frac{1}{2} (x_s^\top Q x_s + u_s^\top R u_s) \!- \!\tau_s \mathcal{R}_s\! \right) \right.\\
 & \left. \left. \quad + \frac{1}{2}x_{N+1}^\top Q_{N+1} x_{N+1} \right | I_t, u_t \right]
\end{aligned}
\end{equation}
be the value function satisfying the recursion
\begin{equation} \label{eq:dp_recursion}
\begin{aligned}
\mathcal{V}_{t} (I_t) 
= \inf_{u_t} 
\sup_{\mathbb{P}\in \mathcal{P_G}(\mathbb{R}^n)} 
\mathbb{E}_{\mathbb{P}}  & \left[ 
\frac{1}{2} (x_t^\top Q x_t + u_t^\top R u_t) - \tau_t \mathcal{R}_t  \right.\\
 & \left. 
 +\mathcal{V}_{t+1}(I_{t+1})  | I_t, u_t \right]
\end{aligned}
\end{equation}
with $\mathcal{V}_{T}(I_{T}) = {\mathbb{E}}_{\mathbb{P}_{w_{T-1}}}[
\frac{1}{2} x_{T}^\top Q_T x_{T} | I_{T} ].$ 
By Bellman's principle of optimality \cite{bertsekas2012dynamic}, it holds
\begin{equation}\label{eq:DP:optimal}
\inf_{\pi}\sup_{\gamma \in \bar{\Gamma}} J^\tau(\pi, \gamma) =  \mathcal{V}_0(I_0).
\end{equation}

Unfortunately, a rapid inspection of \eqref{eq:dp_recursion} reveals that this value function is not amenabe to the closed-form recursion provided by Theorem \ref{propo:recursion}. This is due to the non-linearity introduced by $r_t(\Sigma_t)$ appearing in the recursion for $\mathcal{V}_t(I_t)$. Indeed, $r_t(\Sigma_t)$ is a non-linear function of $\Sigma_t$ which in turns depends on the decision variable $W_{t-1}$ throught the relation\footnote{This relation follows from established principles in filtering theory addressing the propagation of uncertainties in state estimation.}
    \begin{multline*}
        \Sigma_{t} =W_{t-1} + A  \Sigma_{t-1} A^\top - 
        \\ 
         A\Sigma_{t-1} C^\top \left( C \Sigma_{t-1} C^\top + \hat{V}_t\right)^{-1} C \Sigma_{t-1} A^\top, 
    \end{multline*}
    breaking the recursion. To overcome these challenge, we approximate $r_{t}(X)$ in \eqref{eq:r_rec} with a first-order approximation around the nominal prediction error covariance $\hat{\Sigma}_{t}:$ 
     \begin{equation}\label{eq:approx_r_lin}
     \bar{r}_{t}(X ) = r_{t}( \hat{\Sigma}_{t}) + \Tr\left( G_t (X  -  \hat{\Sigma}_{t}) \right)
     \end{equation}
     where
{$$ G_{t} :=  \left. \pdv{r_{t}}{X} ( X)\right|_{X= \hat{\Sigma}_{t}} = \hat{\Sigma}_{t}^{-1} M_{t} A^\top S_{t+1} A M_{t} \hat{\Sigma}_{t}^{-1} $$
with  
$M_{t} =  (\hat{\Sigma}_{t}^{-1} \! + C^{\top} \hat{V}_t^{-1} C)^{-1}.$ }

    Considering  now \eqref{eq:approx_r_lin} leads to the linearized value function 
     \begin{equation}\label{eq:Vtilde_lin_N}
     \begin{aligned}
         \bar \Vc_{t} (I_{t}) &= {\mathbb{E}_{\mathbb{P}_{w_t}}} 
        \left[\frac{1}{2}x_{t}^\top P_{t} x_{t} 
        + \frac{1}{2} \xi_{t}^\top \bar S_{t} \xi_{t} | I_{t} \right] + z_{t} + c_t
     \end{aligned}
     \end{equation}
    where
    \begin{align*}
        \bar S_{t} \!  & \! = Q + A^\top P_{t+1}A -  P_{t}  + \hat{\Sigma}_{t}^{-1} M_{t}  A^\top S_{t} A M_{t} \hat{\Sigma}_{t}^{-1}
        \\
        c_{t} &= r_{t}(\hat{\Sigma}_{t}) - \Tr( G_{t} \hat{\Sigma}_{t} )\\
        & = \Tr \left( S_{t+1} A M_{t} C^{\top} \hat{V}_t^{-1} C  M_{t} A^\top \right) .
    \end{align*} 
   
    It is now easy to see that the value function $\bar \Vc_{t}$ in \eqref{eq:Vtilde_lin_N} satisfies the dynamic programming recursion in Proposition \ref{propo:recursion}, leading us to define the following solution scheme.

    \emph{Solution scheme:} 
    Let $P_{N+1} = Q_{N+1}, \, \bar{S}_{N+1} = 0, \, \bar{z}_{N+1} = 0, \, \bar{c}_{N+1} = 0$ and consider the recursions
    {\luci
    \begin{align} 
    P_t & = Q + A^\top \left( P_{t+1}^{-1} + BR^{-1}B^\top - \frac{\hat{W}_t}{\tau_t} \right)^{-1}A, 
    \label{eq:P_rec_new} \\
    \bar{S}_t  & = \hat{\Sigma}_{t}^{-1} M_{t} A^\top \bar{S}_{t+1} A M_{t} \hat{\Sigma}_{t}^{-1}  + Q + A^\top P_{t+1}A - P_t \label{eq:Sbar_rec_lin} \\
    \bar z_{t} &= \bar z_{t+1} - \frac{\tau_t}{2}\ln \left| I - \frac{\hat{W}_t(P_{t+1} + \bar S_{t+1})}{\tau_t}\right| \label{eq:zbar_rec} \\
    \bar c_t & = \bar c_{t+1} + 
    \Tr \left( 
    \bar{S}_{t+1} A M_{t} C^{\top} \hat{V}_t^{-1} C  M_{t} A^\top     
    \right)  \label{eq:cbar_rec_lin}
    \end{align} 
    }
    for $t = T-1, \dots, 0.$

    Let $\tau = (\tau_0, \dots, \tau_{T-1})$ and assume $\tau \in \mathcal{T} :=  \mathcal{T}_0 \times \cdots \times \mathcal{T}_{T-1} $ with 
    \begin{equation}
    \Tc_t \!:= \! \{ \tau_t \!\in\! \mathbb{R}_{> 0}  : \:     
    \hat{W}_t^{-1 } \!-\! \frac{P_{t+1}}{\tau_t}\!\succ\! 0, \; 
     \hat{W}_t^{-1 } \! -\! \frac{P_{t+1} + \bar S_{t+1}}{\tau_t}\!\succ\! 0 \} .
    \end{equation}
    At each time step $t,$ the \textit{approximated} value function of problem \eqref{eq:problem:relaxed} (computed 
    by using the linearized $\bar{r}_{t}(X )$ in
    \eqref{eq:approx_r_lin})
     is given by
    \begin{equation} \label{eq:Vtilde_lin_t}
    \bar \Vc_t(I_t)  = {\mathbb{E}}_{\mathbb{P}_{w_t}} \!\left[\frac{1}{2}(x_{t}^\top P_{t}x_{t} + \xi_{t}^\top {\bar S}_{t}\xi_{t}) | I_{t} \right] \!+\! \bar z_{t} + \bar c_t
    \end{equation}
    with  $\xi_t = x_t  - \tilde{x}_t $ and $\tilde{x}_t = {\mathbb{E}}^o [x_t | I_t]$ where ${\mathbb{E}}^o $ denotes the expectation with respect to the worst-case distribution $\mathcal{N}(\mu_t^o, W_t^o)$, whose moments are given 
    by \eqref{eq:wc_mu} and \eqref{eq:wc_cov} upon substituting $S_{t+1}$ with $\bar S_{t+1}$. Finally, the optimal control input (for the \textit{approximated} value function) is $$ u^o_t = K^o_t \tilde{x}_t$$ with $K^o_t$ defined in \eqref{eq:drc_policy}.
    Notice that $\tilde{x}_t $ can be recursively computed as 
    \begin{equation}\label{eq:state_estimation}
    \tilde{x}_{t+1} = A\tilde{x}_{t} + Bu_t +  \tilde{\mu}^o_t + L^o_{t+1} (y_{t+1} - C\tilde{x}_{t} )
    \end{equation}
    for $t = 0, \dots T-1$ where the Kalman gain is 
    \begin{equation}\label{eq:Kalman_gain}
    L^o_t = A \Sigma_t^o C^\top (C \Sigma_t^o C^\top + \hat{V}_t)^{-1}.
    \end{equation}
    Here, $\Sigma_{t}^o$ is the worst-case prediction error covariance matrix at time $t$ and it is recursively computed as  
    \begin{equation}\label{eq:wc_cov_estimation}
    \Sigma_{t+1}^o = A \Sigma_t^o A^\top - A \Sigma_t^o C^\top (C \Sigma_t^o C^\top + \hat{V}_t)^{-1} C \Sigma_t^o A^\top + W^o_t. 
    \end{equation}

    \begin{remark}[On approximating $r_{t}(\Sigma_t)$]

    The non-linear dependence of \( r_{t} \) on \( \Sigma_t \)  is also present in other distributionally robust control formulations, such as the data-driven Wasserstein-based approach discussed in \cite{hakobyan2022wasserstein,hakobyan2024wasserstein}. In the search for closed-form recursions, 
    \cite{hakobyan2022wasserstein,hakobyan2024wasserstein}  disregard such dependence. Conversely, we account for it via a suitable approximation, while retaining tractability. For completeness, we show the benefits of our formulation compared to the one in \cite{hakobyan2022wasserstein,hakobyan2024wasserstein} in Appendix E on a simplified setting amenable to both algorithms. Finally, we notice that other linear approximation of $r_t(\Sigma_t)$ are possible; for example, one can consider
    $\bar{\bar r}_t (X) = \frac{1}{2} \Tr(S_{t+1} A X A^\top )  $.
    In this case, since $r_t (X) \leq \bar{\bar r}_t (X)$ for any $X \in \mathbb{S}^n$, the resulting approximated value function $\bar{\bar{\mathcal{V}}}_t$  satisfies $ \mathcal{V}_t (I_t)  \leq \bar{\bar{\mathcal{V}}}_t (I_t)$ for any $ I_t$ and  $t = T, \dots, 0$, returning a valid upper bound for the control problem \eqref{eq:problem:relaxed}.  
    \end{remark}
    

    \subsection{Constrained problem}\label{sec:exact}
    Next, we extend the results of Subsection \ref{sec:approx} to the constrained minimax Problem \ref{eq:problem:DP} resorting to the Lagrange duality theory. Let $\tau = (\tau_0,\ldots, \tau_{T-1})$ be the Lagrangian multipliers vector; the dual problem associated to \eqref{eq:problem:DP}  is\footnote{Contrary to Subsection \ref{sec:approx}, we will consider optimizing over $\tau$, hence we use the notation $\mathcal{V}_t(I_t, \tau)$ to emphasize such dependence.}
    \begin{align*}
    & \inf_{\tau \in \R_+^n} 
    \inf_{\pi \in \Pi} 
    \sup_{\gamma \in \Gamma} 
    J(\pi, \gamma)  - \sum_{t=0}^{N} \tau_t (\Rc_t - \rho_t)  \\
    & = 
    \inf_{\tau \in \R_+^n} \Vc_0(I_0, \tau)  + \sum_{t=0}^{N} \tau_t  \rho_t
    \end{align*}
    where the equivalence follows from \eqref{eq:problem:DP}. 
    We consider the linearized dual problem
    \begin{equation} \label{pb:dual_approx}
    \inf_{\tau \in \R_+^n} \bar \Wc_0 (I_0, \tau) := \bar \Vc_0(I_0, \tau)  + \sum_{t=0}^{N} \tau_t  \rho_t
    \end{equation}
    with $ \bar \Vc_0 (I_0, \tau)$ defined in \eqref{eq:Vtilde_lin_t}.
    
    Consider the coordinate gradient descent scheme, where at each iteration $t \in \mathbb{N}$ we update each $i$-th component of the vector $\tau$ sequentially\footnote{One might also randomly select the order of the updates at each iteration $t$ rather than considering a cyclic pattern. Randomized updates might enhance numerical stability.} for $i = \{0,\ldots, T-1\}$ by solving the one-dimensional subproblem
    \begin{equation}\label{CGD}
        \tau_i^{t+1} = \argmin_{\xi\in \R_+} \bar \Wc_0(I_0, \tau_{0}^{t+1},, \ldots, \tau_{i-1}^{t+1}, \xi,  \tau^{t}_{i+1}, \ldots, \tau_T^t).
    \end{equation}
    
    The overall procedure to solve \eqref{eq:problem:DP} is summarized in Algorithm 1.  We get the following result.

    \begin{theorem}\label{CGD:convergence}
    Let $\tau^t$ be the sequence generated by the coordinate gradient descent scheme ta time $t$ using the updates in \eqref{CGD}. Then, the sequence $\{ \bar \Wc_0(I_0, \tau^t) \}_{t=0}^{\infty}$  converges to the optimal value of the dual problem \eqref{pb:dual_approx} at a rate no worse than $\mathcal{O}\left( \frac{1}{t} \right).$
    \end{theorem}

\begin{algorithm}[t!]\label{alg:CGD}
    \caption{DR-LQG with endogenous ambiguity sets}\label{alg}
    \begin{algorithmic}
    \Require{ \texttt{sys},$\hat{V}_t, \hat{W}_t, \{\rho_t\}_{t=0}^{T-1}$} 
    \vskip 0.2cm
    \State \textbf{Offline part:}
    \State Initialize $\{\tau_t^{(0)}\}_{t=0}^{T-1}$ and set $j = 0$
    \Repeat   
            \State 
            $\bar{\mathcal{W}}_0(I_0, \{\tau_t^{(j)}\}_{t=0}^{T-1}) \: \leftarrow$ \texttt{Recurs}(\texttt{sys}, $\{\tau_t^{(j)}\}_{t=0}^{T-1}$)
            \For{$i = 0,\ldots, T-1$}
                \State                 {\small
                $
                        \tau_i^{(j+1)} \!\!\coloneqq \!\argmin_z \: \bar{\mathcal{W}}_0\left(\!I_0,\! \left[\{\tau_t^{(j)}\}_{t=0}^{i-1},z,\!\{\tau_t^{(j)}\}_{t=i+1}^{T-1}\right]\right)
                $
                 }            
            \EndFor
            \State Set $j = j+1$
    \Until{convergence} 
    \State Set $(\tau_0^\star, \ldots, \tau_{T-1}^\star) = (\tau_0^{(j)}, \ldots, \tau_{T-1}^{(j)})$
    \State Compute $K_t^\star, L_t^\star, {\mu}_t^\star, {W}_t^\star$ based on $\tau^\star$
    \vskip 0.2cm
    \State \textbf{Online part:}
    \For{$t=N$ to $0$} 
            \State Apply $u_t^\star = K_t^\star \tilde{x}_t$ to system \eqref{eq:nominal:system}
            \State Measure $y_{t+1}$ and estimate $\tilde{x}_{t+1}$ via \eqref{eq:state_estimation}
    \EndFor
    \end{algorithmic}
    \end{algorithm}
    
    \begin{algorithm}[t!]
    \caption{\texttt{Recursions}} \label{alg2}
    \begin{algorithmic}
    \Require system/cost matrices, $\hat{V}_t, \hat{W}_t, \rho_t, 
    \{\tau_t^{(j)}\}_{t=0}^N$ 
    \Ensure{total cost $\bar{\mathcal{W}}_0(I_0, \{\tau_t^{(j)}\}_{t=0}^{T-1})$} 
    \vskip 0.2cm
    \For{ $t=T$ to $t=0$} \Comment{Backward pass}
        \State Compute recursion matrices via \eqref{eq:P_rec_new}-\eqref{eq:cbar_rec_lin}
        \State Compute control gain via \eqref{eq:drc_policy}
        \State Compute worst-case distribution via \eqref{eq:wc_mu}, \eqref{eq:wc_cov}
    \EndFor
    \For{ $t=0$ to $t=T$ } \Comment{Foreward pass}
        \State Compute Kalman gain via \eqref{eq:Kalman_gain}
        \State Compute worst-case prediction covariance via \eqref{eq:wc_cov_estimation}
    \EndFor
    \State Compute $\bar{\mathcal{W}}_0(I_0, \{\tau_t^{(j)}\}_{t=0}^{T-1}) \coloneqq\bar{\mathcal{V}}_0(I_0) + \sum_{t=0}^{T-1} \tau_t^{(j)} \rho_t$
    \end{algorithmic}
    \end{algorithm}

    \section{Simulations}\label{sec;Simulations}
    
    \subsection{Example 1} We consider a linear model \eqref{eq:nominal:system} with 
    $$
    A =     \begin{bmatrix} 
    1.1 & 0.1 \\ 0 & 0.95
     \end{bmatrix},
    \; 
    B =  \begin{bmatrix}     0.2 \\ 1    \end{bmatrix},  
    \;
    C  =    \begin{bmatrix} 1 & 0 \end{bmatrix} .
    $$
    We set the radii to be $\rho_{x_0} = \rho_{w_t} = \rho_{v_t} = 1$ and the nominal covariances to $\hat{W}_t = 0.001 I_2$, $\hat{V}_t = 0.001$ for all times $t$, and $W_{-1} = 0_{2}$. We set $Q = I_2, Q_t = 10 I_2, R = 0.1$, $T = 20$, and $\hat{x_0} = [ 0, 0]^\top$. We implement the iterative best response dynamics in \eqref{eq:BR:continuous} in Python 3.8.6 using the Scipy package and the RK45 ODE solver. Fig.~\ref{fig:ex1:br} shows the empirical convergence behavior of the best response dynamics, confirming the exponential convergence rate from Theorem \ref{thr:convergence:CT}. 
    

    The DR-LQG controller was benchmarked against a standard LQG based on the nominal distributions. We carried out 5000 Montecarlo simulations with exogenous disturbances distributed according to the true distributions $\mathbb{P}_{x_0}, \mathbb{P}_{w_t}, \mathbb{P}_{v_t}$ selected randomly from the ambiguity sets. The DR-LQG controller led to an average cost across the Montecarlo run of 0.6287 and a standard deviation of 0.7125, while the standard LQG led to an average cost of 0.6587 and a standard deviation of 0.7588. The results confirm the effectiveness of the proposed approach to provide robustification against distributional ambiguity.

    \begin{figure}
        \centering	\includegraphics[width=0.85\linewidth]{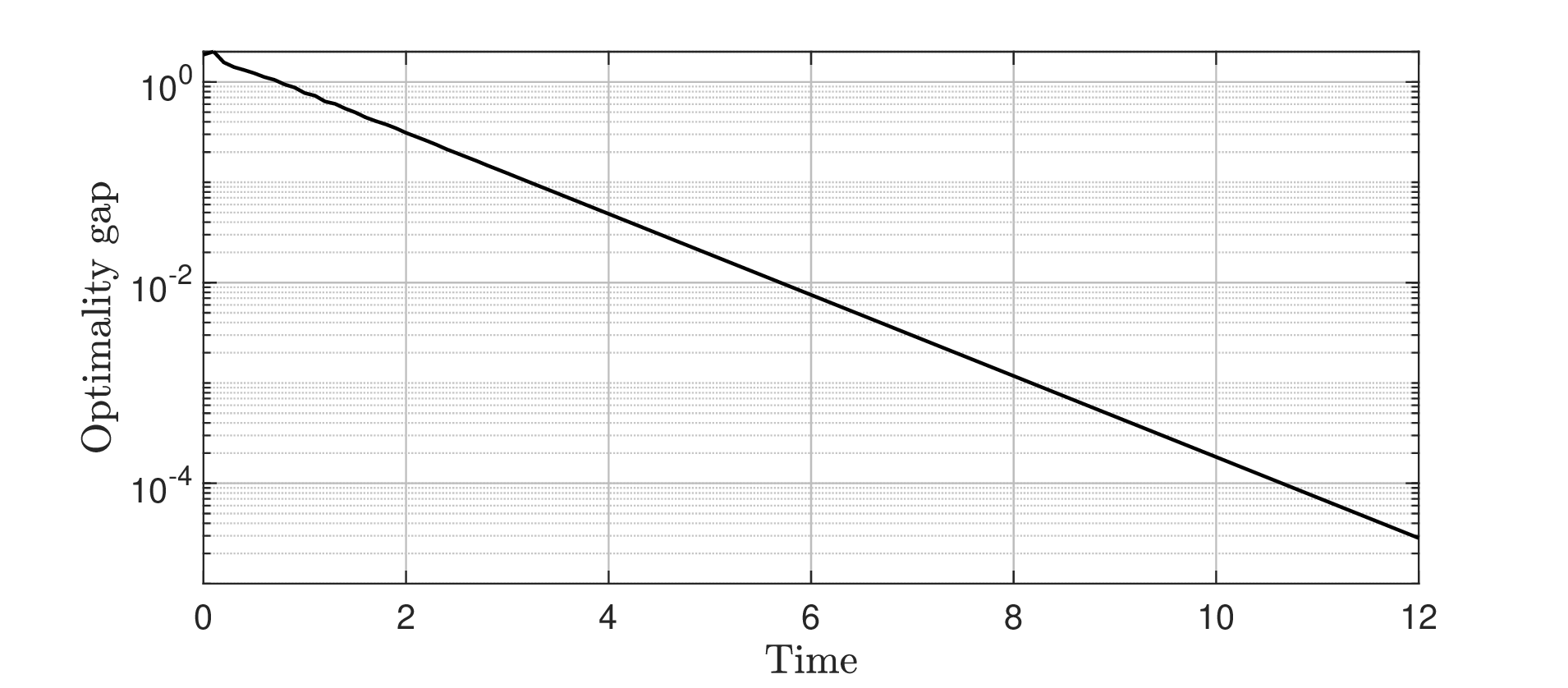}
        \caption{Example 1: Convergence of the optimality gap for the best response dynamics.}
        \label{fig:ex1:br}
    \end{figure}
    
    \subsection{Example 2}
    We show the benefits of our DR-LQG controller with endogenous ambiguity sets \eqref{eq:problem:DP}. Consider the linear model \eqref{eq:nominal:system} with
    $$
    A =     \begin{bmatrix} 
    1.05 & -0.05 \\ 0.5 & 1.05
     \end{bmatrix},
    \; 
    B =  \begin{bmatrix}     0.2157  \\ 0.2367    \end{bmatrix},  
    \;
    C  =    \begin{bmatrix} -200 & 100 \end{bmatrix} .
    $$
    We assume that, for all $t$, $w_t$ follows a Gaussian distribution with zero-mean and nominal covariance  
    $$
     \hat{W}_t  =  \begin{bmatrix} 0.047 & 0.095 \\  0.095 & 0.2  \end{bmatrix},
    $$
   while $v_t$ follows a zero-mean Gaussian distribution with covariance $ \hat{V}_t = I.$  We set $\hat x_0 = [0.1 \;\;  0.1 ]^\top, W_{-1} = 10^{-4} I_2.$ 
    We consider a horizon $T= 50$ and use the cost matrices
    $$ 
    Q = 
    \begin{bmatrix} 
        0.25 & -0.25 \\ -0.25 & 0.25
    \end{bmatrix}, \; \;
     Q_{T} = 10^{-3} 
    \begin{bmatrix} 
        0.625 & -0.275 \\ -0.275 & 0.125
    \end{bmatrix}
    $$
    and $R =  10^{-3} I.$ We assume that the matrix $A$ contains some uncertainty and that the real underlying system evolves with a dynamics matrix $\tilde A =  A + \Delta A $ where 
    $$ \Delta A  = 
     \begin{bmatrix} 
     a & b \\ 
     0 & a 
     \end{bmatrix}.$$ 
    The coefficients $a$ and $b$ are unknown parameters with
    $ |a| \leq a_{M} := 0.047 $ 
    and 
    $ | b | \leq b_{M} := 0.03.$
    In other words, the real underlying system evolves according to the dynamics 
    $x_{t+1} = A x_t + B u_t + \tilde w_t $
    where $\tilde w_t = \Delta A x_t + w_t \sim \mathcal{N} (\Delta A x_t , \hat{W}). $ 
    For the ambiguity set, we use $z_t = E_1 x_t $ with  
    $$E_1 =  V^{-\frac{1}{2}} 
    \begin{bmatrix}  a_{M} & b_{M} \\  0 & a_{M}  \end{bmatrix}
    = 
    \begin{bmatrix}  0.996 & 0.204 \\  -0.431 & 0.026  \end{bmatrix}.
    $$ 
    and  $\rho_t = 10^{-5}.$ We compare our controller, which we term D$^2$O-LQG controller, with the standard LQG controller. Additionally, we consider the DRC controller with a single relative-entropy constraint (D-LQG) from \cite{Petersen2000}, using $\rho = \sum_{t=0}^{T-1} \rho_t.$ We consider two scenarios: (i)  Nominal scenario, with $ \Delta A = 0$; (ii) Perturbed scenario, with $a = 0.03$ and $b=0.02.$ The results of the simulations are summarized in Figg.~\ref{fig:ex2_nominal} and \ref{fig:ex2_pertubed}. 
    The standard LQG technique is not able to stabilize the system in the perturbed scenario, causing the cost index to increase dramatically. The D-LQG controller with a single constraint has not a satisfactory behaviour: this is due to the fact that  the maximizing player is allowed to allocate most of the mismatch budget to few (or even one) time intervals. On the other hands, D$^2$O-LQG control in able to trade off optimality and robustness. In the nominal scenario the average closed-loop cost is slightly larger than the pure LQG optimum. This cost remains almost constant when $ \Delta A \neq 0,$ giving evidence to the robustness properties of the control system.

    \begin{figure} 
        \centering	\includegraphics[width=0.75\linewidth]{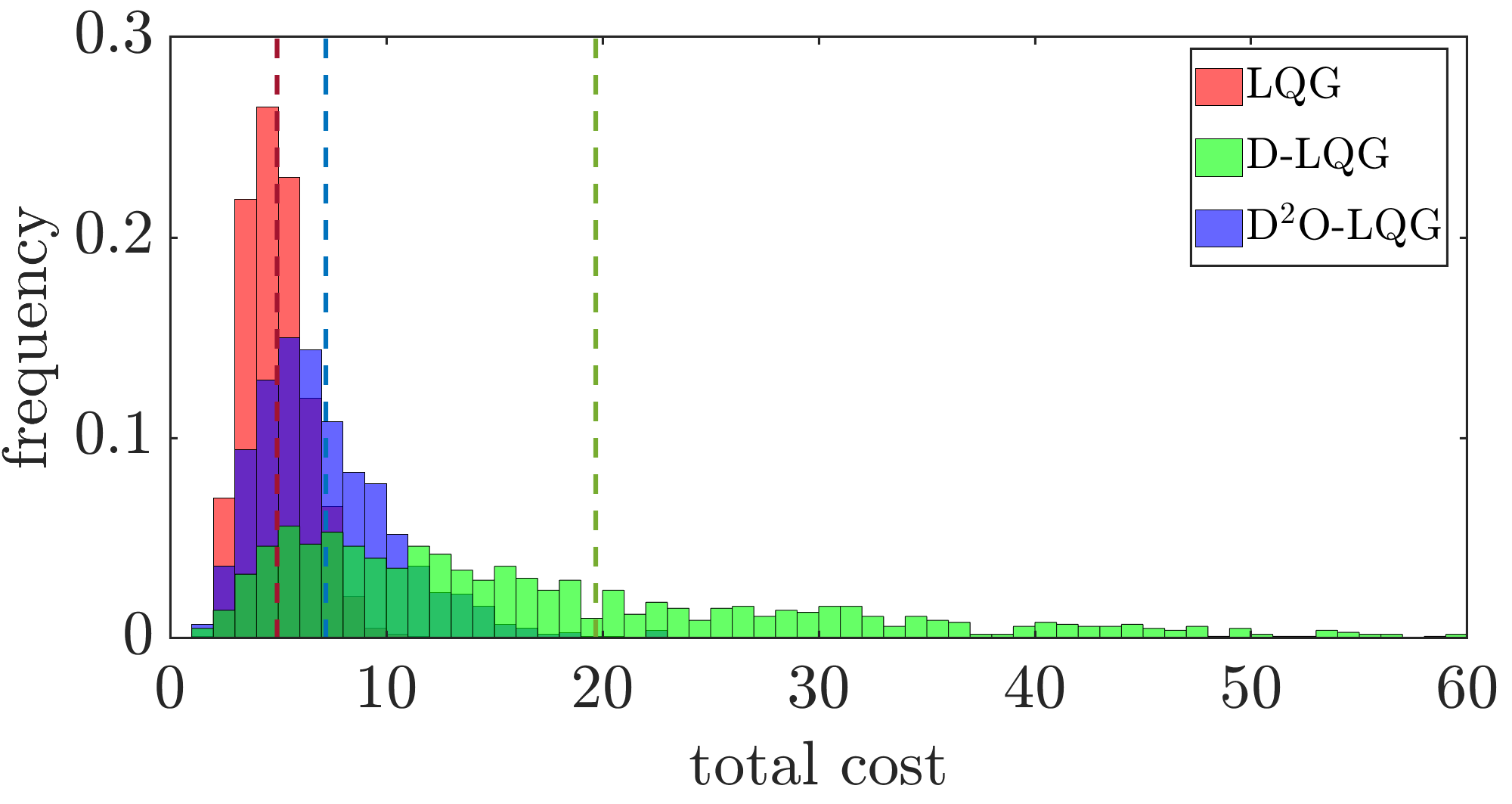}
        \caption{Example 2: nominal scenario with  $ \Delta A = 0.$ Histogram of the cost distribution across 1000 Monte Carlo experiments. The dashed lines represent the sample means of the costs. }
        \label{fig:ex2_nominal}
    \end{figure}
    
    \begin{figure} 
        \centering	\includegraphics[width=\linewidth]{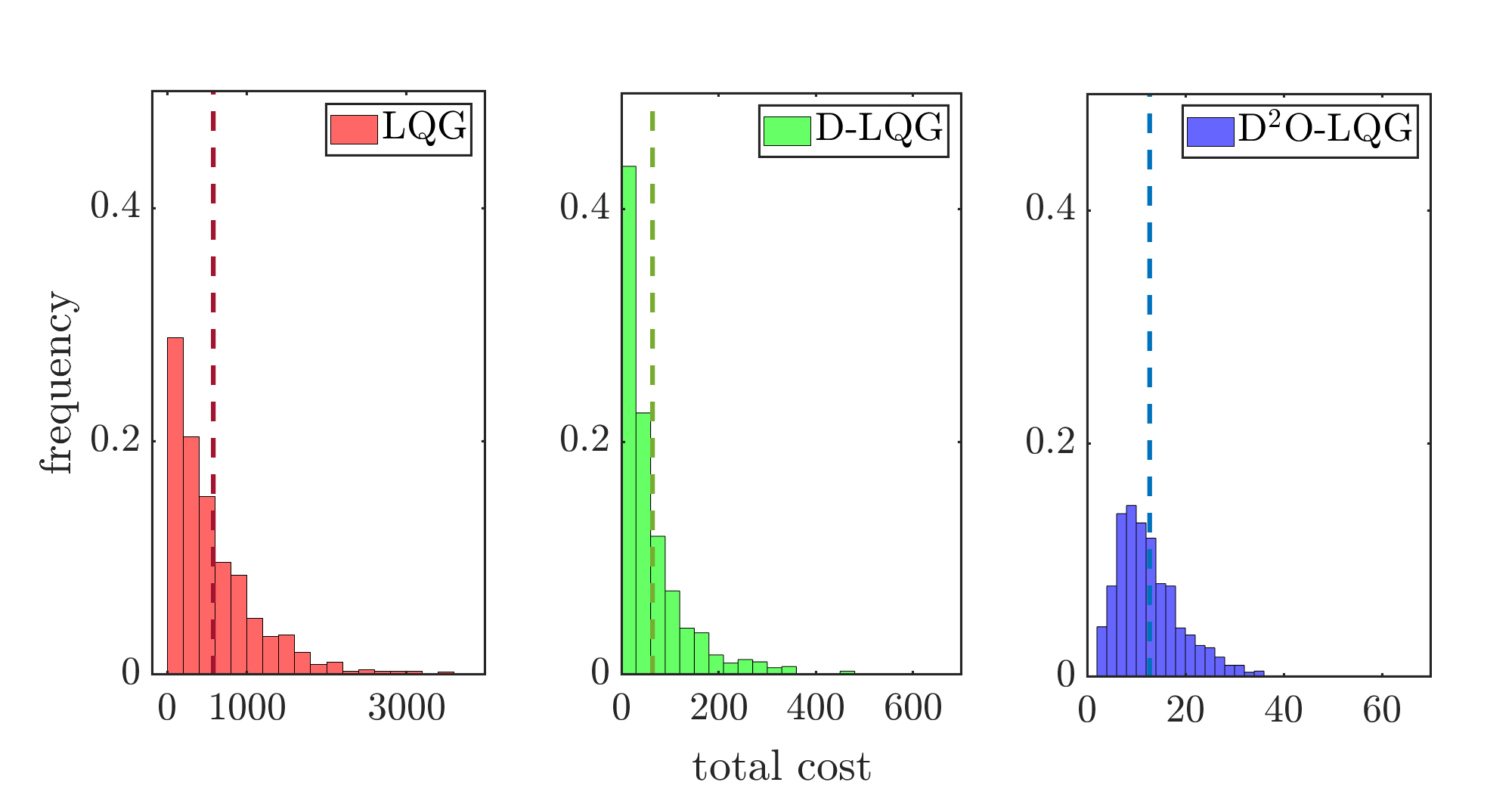}
            \caption{Example 2: perturbed scenario with  $ \Delta A \neq 0.$ Histogram of the cost distribution across 1000 Monte Carlo experiments. The dashed lines represent the sample means of the costs. Notice the different $x$-axis scale.}
        \label{fig:ex2_pertubed}
    \end{figure}
    
    \subsection{Example 3}
    In the last example, we aim to show that the scheme proposed in Section \ref{sec:endo} outperforms the one in \cite{hakobyan2024wasserstein} based on the Wasserstein distance. As \cite{hakobyan2024wasserstein} does not directly handle decision-dependent ambiguity sets, we consider a simplified setting with exogenous uncertainty and adapt the recursions in Proposition \eqref{propo:recursion} accordingly by setting $E_1 = E_2 = 0$. We consider a building temperature control problem using a state-space model borrowed from \cite{hewing2020recursively} and affected by uncertainty in the ambient temperature $T^a_t \sim \mathcal{N}(\bar{T}^a_t + \hat{\mu}_t, \hat{W}_t)$\footnote{The recursions in this case are slightly different; but can be easilly derived using the same reasoning as in Proposition \ref{propo:recursion}.}. The control problem seeks to regulate the rooms temperature to a reference $r = 21^\circ$C at the minimum power consumption via the stage cost $\|x_t - r\|_Q^2 + \|u_t\|_R^2$.

    To fit the building temperature control problem into the presented framework, we define the error state $e_{t} = x_t - r$ and consider the error dynamics
    \begin{equation}\label{eq:building2}
    \begin{aligned}
   e_{t+1} & = A e_t + Bu_t + w_t,\\
   y_t & = Ce_t + v_t,
   \end{aligned}
   \end{equation}
    with $w_t \sim \mathcal{N}(F\bar{w}_t + (Ar - r) + F\hat{\mu}_t, F\hat{W}_t F^\top)$, $v_t \sim \mathcal{N}(Cr, V)$, and $e_0 \sim \mathcal{N}(\hat{x}_0 - r, \Sigma_0)$.  We let $W = I_2$, $\hat{x}_0 =\begin{bmatrix}20.5 & 19.75 & 20 & 20.2\end{bmatrix}^\top$ and $\Sigma_0 = 0.1I_2$. {\marti For the sake of the simulation, we set the uncertainty budget $\rho_t$ a-posteriori based on the knowledge of the true process noise distribution, considering $d_t = 1.1 \mathcal{R}_t$ where $\mathcal{R}_t$ is the relative entropy between the nominal and the true distribution at time $t$. We compare the proposed D$^2$O-LQG controller with the standard LQG controller and the DRC controller with a constant distributional ambiguity budget $\theta$ per each time step proposed in \cite{hakobyan2024wasserstein}\footnote{We use the public code from the authors accessible at \url{https://github.com/CORE-SNU/PO-WDRC}.} (W-DRC). As before, we set $\theta = \frac{1}{N}\sum_{t=0}^{N-1}(1.1 d^{\text{W}}_t)$, where $d^{\text{W}}_t$ is the a-posteriori Wasserstein distance between the true and the nominal distribution.  Results are summarized in Fig.~\ref{fig:ex1}. Again, the standard LQG controller does not offer any robustness against model misspecification; thus, its performance rapidly deteriorates in the presence of distributional ambiguity. On the other hand, while the W-DRC controller from \cite{hakobyan2024wasserstein} offers a degree of robustness, its practical performance is hindered by the requirement for the same ambiguity set size (e.g., the same $\theta$) throughout the entire control task.
   
   \begin{figure} 
       \centering	\includegraphics[width=0.75\linewidth]{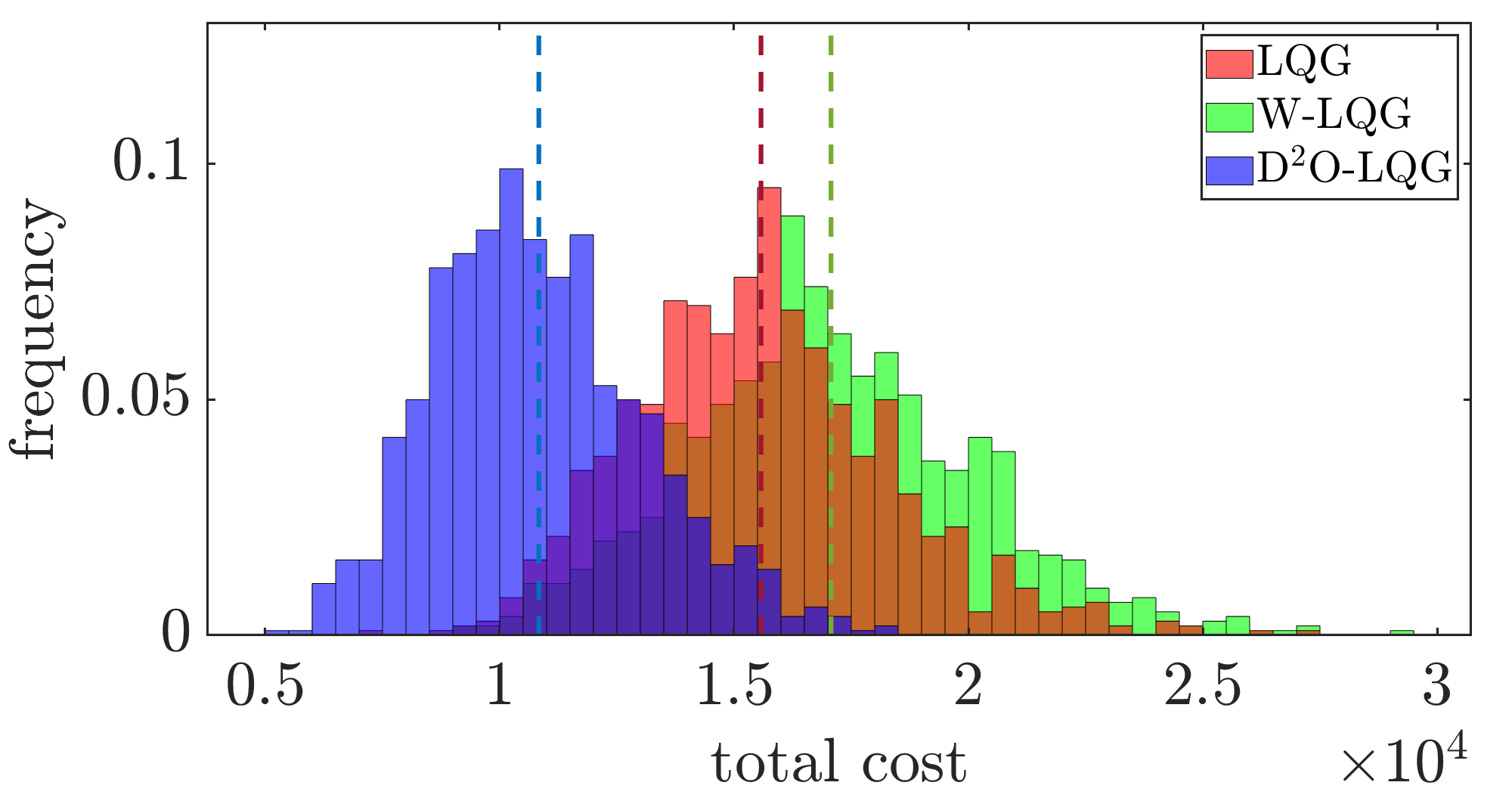}
       \caption{Example 1. Histogram of the cost distribution across 1000 Monte Carlo experiments. The dashed lines represent the sample means of the costs.}
       \label{fig:ex1}
   \end{figure}


    \section{Conclusions}\label{sec:conclusions}
    For discrete-time stochastic linear systems, we propose an output feedback controller capable of robustifying the standard LQG approach against distributional ambiguity affecting both process and measurement noise by relying on KL ambiguity sets. Our analysis shows that linear policies are still optimal despite the added complexity; moreover, the worst-case distribution is still a Gaussian. These insights led us to design an iterated best response dyanmics scheme that provably convergences to the set of saddle points and admits closed-form expressions. Further, we consider the case of decision-dependent ambiguity sets to capture model perturbations. For this setting, we devise a tailored approximated recursive scheme based on dynamic programming and coordinate gradient descent.

    
    \bibliographystyle{ieeetr}        
    \bibliography{references_new}           

\begin{thebibliography}{10}

\bibitem{doyle1978guaranteed}
J.~Doyle, ``Guaranteed margins for lqg regulators,'' {\em IEEE Transactions on
  automatic Control}, vol.~23, no.~4, pp.~756--757, 1978.

\bibitem{taskesen2023distributionally}
B.~Taskesen, D.~Iancu, {\c{C}}.~Ko{\c{c}}yi{\u{g}}it, and D.~Kuhn,
  ``Distributionally robust linear quadratic control,'' {\em Advances in Neural
  Information Processing Systems}, vol.~36, pp.~18613--18632, 2023.

\bibitem{lanzetti2024optimality}
N.~Lanzetti, A.~Terpin, and F.~D{\"o}rfler, ``Optimality of linear policies for
  distributionally robust linear quadratic gaussian regulator with stationary
  distributions,'' {\em arXiv preprint arXiv:2410.22826}, 2024.

\bibitem{hassibi1999indefinite}
B.~Hassibi, A.~H. Sayed, and T.~Kailath, {\em Indefinite-Quadratic estimation
  and control: a unified approach to $H_2$ and $H_\infty$ theories}.
\newblock SIAM, 1999.

\bibitem{zhou1998essentials}
K.~Zhou and J.~C. Doyle, {\em Essentials of robust control}, vol.~104.
\newblock Prentice hall Upper Saddle River, NJ, 1998.

\bibitem{jacobson1973optimal}
D.~Jacobson, ``Optimal stochastic linear systems with exponential performance
  criteria and their relation to deterministic differential games,'' {\em IEEE
  Transactions on Automatic control}, vol.~18, no.~2, pp.~124--131, 1973.

\bibitem{whittle1981risk}
P.~Whittle, ``Risk-sensitive linear/quadratic/gaussian control,'' {\em Advances
  in Applied Probability}, vol.~13, no.~4, pp.~764--777, 1981.

\bibitem{van2015distributionally}
B.~Van~Parys, D.~Kuhn, P.~Goulart, and M.~Morari, ``Distributionally robust
  control of constrained stochastic systems,'' {\em IEEE Transactions on
  Automatic Control}, vol.~61, no.~2, pp.~430--442, 2015.

\bibitem{hakobyan2022wasserstein}
A.~Hakobyan and I.~Yang, ``Wasserstein distributionally robust control of
  partially observable linear systems: Tractable approximation and performance
  guarantee,'' in {\em 2022 IEEE 61st Conference on Decision and Control
  (CDC)}, pp.~4800--4807, IEEE, 2022.

\bibitem{fochesato2022data}
M.~Fochesato and J.~Lygeros, ``Data-driven distributionally robust bounds for
  stochastic model predictive control,'' in {\em 2022 IEEE 61st Conference on
  Decision and Control (CDC)}, pp.~3611--3616, IEEE, 2022.

\bibitem{aolaritei2023wasserstein}
L.~Aolaritei, M.~Fochesato, J.~Lygeros, and F.~D{\"o}rfler, ``Wasserstein tube
  mpc with exact uncertainty propagation,'' in {\em 2023 62nd IEEE Conference
  on Decision and Control (CDC)}, pp.~2036--2041, IEEE, 2023.

\bibitem{petersen2000minimax}
I.~R. Petersen, M.~R. James, and P.~Dupuis, ``Minimax optimal control of
  stochastic uncertain systems with relative entropy constraints,'' {\em IEEE
  Transactions on Automatic Control}, vol.~45, no.~3, pp.~398--412, 2000.

\bibitem{hakobyan2024wasserstein}
A.~Hakobyan and I.~Yang, ``Wasserstein distributionally robust control of
  partially observable linear stochastic systems,'' {\em IEEE Transactions on
  Automatic Control}, 2024.

\bibitem{brouillon2023distributionally}
J.-S. Brouillon, A.~Martin, J.~Lygeros, F.~D{\"o}rfler, and G.~F. Trecate,
  ``Distributionally robust infinite-horizon control: from a pool of samples to
  the design of dependable controllers,'' {\em arXiv preprint
  arXiv:2312.07324}, 2023.

\bibitem{falconi2025distributionally}
L.~Falconi, A.~Ferrante, and M.~Zorzi, ``Distributionally robust lqg control
  under distributed uncertainty,'' {\em Automatica}, vol.~174, p.~112128, 2025.

\bibitem{tzortzis2015dynamic}
I.~Tzortzis, C.~Charalambous, and T.~Charalambous, ``Dynamic programming
  subject to total variation distance ambiguity,'' {\em SIAM Journal on Control
  and Optimization}, vol.~53, no.~4, pp.~2040--2075, 2015.

\bibitem{mackay2003information}
D.~J. MacKay, {\em Information theory, inference and learning algorithms}.
\newblock Cambridge university press, 2003.

\bibitem{boyd:vandenberghe:2004}
S.~Boyd and L.~Vandenberghe, {\em Convex optimization}.
\newblock Cambridge, United Kingdom: Cambridge University Press, 2004.

\bibitem{ljung1999system}
L.~Ljung, {\em System identification: Theory of the User}.
\newblock Prentice Hall Information and System Sciences Series, Thomas Kailath,
  Editor, 1999.

\bibitem{akella2024risk}
P.~Akella, A.~Dixit, M.~Ahmadi, L.~Lindemann, M.~P. Chapman, G.~J. Pappas,
  A.~D. Ames, and J.~W. Burdick, ``Risk-aware robotics: Tail risk measures in
  planning, control, and verification,'' {\em arXiv preprint arXiv:2403.18972},
  2024.

\bibitem{shapiro2021lectures}
A.~Shapiro, D.~Dentcheva, and A.~Ruszczynski, {\em Lectures on stochastic
  programming: modeling and theory}.
\newblock SIAM, 2021.

\bibitem{chapman2021toward}
M.~P. Chapman and L.~Lessard, ``Toward a scalable upper bound for a cvar-lq
  problem,'' {\em IEEE Control Systems Letters}, vol.~6, pp.~920--925, 2021.

\bibitem{tsiamis2021linear}
A.~Tsiamis, D.~S. Kalogerias, A.~Ribeiro, and G.~J. Pappas, ``Linear quadratic
  control with risk constraints,'' {\em arXiv preprint arXiv:2112.07564}, 2021.

\bibitem{hadjiyiannis2011efficient}
M.~J. Hadjiyiannis, P.~J. Goulart, and D.~Kuhn, ``An efficient method to
  estimate the suboptimality of affine controllers,'' {\em IEEE Transactions on
  Automatic Control}, vol.~56, no.~12, pp.~2841--2853, 2011.

\bibitem{kim2023distributional}
K.~Kim and I.~Yang, ``Distributional robustness in minimax linear quadratic
  control with wasserstein distance,'' {\em SIAM Journal on Control and
  Optimization}, vol.~61, no.~2, pp.~458--483, 2023.

\bibitem{bertsekas2012dynamic}
D.~Bertsekas, {\em Dynamic programming and optimal control: Volume I}, vol.~4.
\newblock Athena scientific, 2012.

\bibitem{Petersen2000}
I.~Petersen, M.~James, and P.~Dupuis, ``Minimax optimal control of stochastic
  uncertain systems with relative entropy constraints,'' {\em IEEE Transactions
  on Automatic Control}, vol.~45, no.~3, pp.~398--412, 2000.

\bibitem{hewing2020recursively}
L.~Hewing, K.~P. Wabersich, and M.~N. Zeilinger, ``Recursively feasible
  stochastic model predictive control using indirect feedback,'' {\em
  Automatica}, vol.~119, p.~109095, 2020.

\bibitem{boyd2004convex}
S.~P. Boyd and L.~Vandenberghe, {\em Convex optimization}.
\newblock Cambridge university press, 2004.

\bibitem{rockafellar1997convex}
R.~T. Rockafellar, {\em Convex analysis}, vol.~28.
\newblock Princeton university press, 1997.

\bibitem{sion1958general}
M.~Sion, ``On general minimax theorems.,'' 1958.

\bibitem{hofbauer2006best}
J.~Hofbauer and S.~Sorin, ``Best response dynamics for continuous zero-sum
  games,'' {\em Discrete and Continuous Dynamical Systems Series B}, vol.~6,
  no.~1, p.~215, 2006.

\bibitem{nocedal1999numerical}
J.~Nocedal and S.~J. Wright, {\em Numerical optimization}.
\newblock Springer, 1999.

\end{thebibliography}
    

    \section*{Appendix A: System matrices}
    The stacked cost matrices are defined as $\mathbf{Q} = \text{diag}(Q,\ldots, Q, Q_T) \in \mathbb{S}_+^{n(T+1)},\: \mathbf{R} = \text{diag}(R, \ldots, R) \in \mathbb{R}_{++}^{mT}$. The stacked system matrices $\mathbf{C} \in \mathbb{R}^{p T \times n(T+1)}, \bG \in \mathbb{R}^{n(T+1) \times n(T+1)}$ and $\bH \in \mathbb{R}^{n(T+1) \times m T}$ are defined as

$$
\bC\!=\!\left[\!\begin{array}{ccccc}
    C & 0 &        &    & \\
      & C & 0      &    & \\
      &   & \ddots &    & \\
      &   &        & C  & 0
    \end{array}\!\right]\!, \: \bG\!=\!\left[\!\begin{array}{cccc}
A^0 & & & \\
A^1 & A^1 & & \\
\vdots & & \ddots & \\
A^T & A^T & \ldots & A^T
\end{array}\!\right]
$$

and

$$
\bH=\left[\begin{array}{ccccc}
0 & & & & \\
A^1 B & 0 & & & \\
A^2 B & A^2 B & 0 & & \\
\vdots & & & \ddots & \\
\vdots & & & & 0 \\
A^T B & A^T B & \ldots & \ldots & A^T B
\end{array}\right].
$$

    \section*{Appendix B: Auxiliary results}
    The following result  is proven in \cite{hadjiyiannis2011efficient, taskesen2023distributionally}.
    \begin{lemma}\label{lemma:change}
        If $\bu=\bU \boldeta+\bq$ for some $\bU \in \mathbb{R}^{mT \times pT}$ and $\bq \in \mathbb{R}^{m T}$, then $\bu=\bU^{\prime} \by+\bq^{\prime}$ for $\bU^{\prime}=(I+\bU \bC \bH)^{-1} \bU$ and $\bq^{\prime}=(I+\bU \bC \bH)^{-1} \bq$. Conversely, if $\bu=\bU^{\prime} \by+\bq^{\prime}$ for some $\bU^{\prime}\in \mathbb{R}^{mT \times pT}$ and $\bq^{\prime} \in \mathbb{R}^{m T}$, then $\bu=\bU \boldeta+\bq$ for $\bU=\left(I-\bU^{\prime} \bC \bH\right)^{-1} \bU^{\prime}$ and $\bq=\left(I-\bU^{\prime} \bC \bH\right)^{-1} \bq^{\prime}$.
    \end{lemma}
    
    
    

    

    \begin{lemma}\label{G:nice}
        Let $\hat{\Sigma} \succ 0, \rho \geq 0$. The set $\mathcal{G} := \{\Sigma \succ 0\; : \; \frac{1}{2} \left( \emph{Tr}(\hat{\Sigma}^{-1} \Sigma) - n + \log \frac{\det \hat{\Sigma}}{\det \Sigma} \right) \leq \rho\}$ is compact and convex.
        \end{lemma}
        \begin{proof}
        \textit{Convexity:} Let $f(\Sigma) =  \text{Tr}(\hat{\Sigma}^{-1} \Sigma) - \log\det{\Sigma}$; hence, $f(\Sigma) \leq \tilde{\rho}$ with $\tilde{\rho} = 2 \left(n - \log\det{\hat{\Sigma} + \rho}\right)$. Note that $f(W\Sigma$ is convex as it is the sum of two convex functions: (i) the trace is linear in $\Sigma$ and (ii) $-\log\det\Sigma$ \cite[pag.~74]{boyd2004convex}. Finally, $\mathcal{G}$ is convex since the sublevel sets of a convex function are convex \cite[pag.~75]{boyd2004convex}.
        
        \textit{Compactness:} We prove (i) Closedness: the function $f(\Sigma)$ is continuous in $\Sigma$ over the set of positive definite matrices. The sublevel sets of a continuous function are closed \cite{rockafellar1997convex}; (ii) Boundedness: Let $\Sigma = U\Lambda U^\top$ be the eigenvalue decomposition of $\Sigma$, where $\Lambda = \text{diag}(\lambda_1, \ldots, \lambda_d)$. Then, denoting as $\mu_i$ the $i$-th eigenvalue of $\hat{\Sigma}^{-1}$, we have
        $$
         f(\Sigma) = \sum_{i=1}^d\mu_i \lambda_i - \sum_{i=1}^d \log \lambda_i
        \implies \sum_{i=1}^d\mu_i \lambda_i - \sum_{i=1}^d \log \lambda_i \leq \tilde{\rho}.
        $$
        Since $\log \lambda_i \geq 0$ we get 
        $$
        \sum_{i=1}^d\mu_i \lambda_i \leq \tilde{\rho} \implies \sum_{i=1}^d \lambda_i \leq \frac{\tilde{\rho}}{\min_i \mu_i}.
        $$
        That is, the sum of the eigenvalues of $\Sigma$ is bounded. This implies that the eigenvalues of $\Sigma$ are bounded, and hence $\Sigma$ is bounded. Finally, $\mathcal{G}$ is compact since it is closed and bounded.
        \end{proof}
                          
        \begin{lemma}\label{lemma:opt:variance}
            Let $\hat{\Sigma }\succ 0$. Consider the following optimization problem parametric in $M \in \mathbb{S}^n_{+}$ and $\tau \in \mathbb{R}_+$
           $$
           J = \sup_{\Sigma \succ 0} \frac{1}{2}\emph{Tr}(M{\Sigma}) - \frac{\tau}{2}\left(\emph{Tr}(\hat{\Sigma}^{-1}\Sigma) - n + \log \frac{|\hat{\Sigma}|}{|\Sigma|} \right),
           $$ 
           and assume $(\hat{\Sigma}^{-1} - \frac{M}{\tau}) \succ 0$. Then, $
                   \Sigma^\star = {\left(\hat{\Sigma}^{-1} - \frac{M} {\tau}\right)^{-1} },$ and 
                   $
                   J^o = - \frac{\tau}{2}\ln \left| I - \frac{VM}{\tau} \right|. 
                   $
           \end{lemma}
           
           \begin{proof}
           First, we notice that the objective function is strictly concave in $\Sigma$ under the assumption $(\hat{\Sigma}^{-1} - \frac{M}{\tau}) \succ 0$. By \cite{boyd2004convex}, the optimum is obtained by annihilating the first derivative
           $$
           \begin{aligned}
           & \frac{\partial J(\Sigma)}{\partial \Sigma} = \frac{1}{2}M - \frac{\tau}{2}(\hat{\Sigma}^{-1} - \Sigma^{-1}) = 0 \\
           \implies &  \Sigma^\star =\textstyle \left(\hat{\Sigma}^{-1} - \frac{M}{\tau} \right)^{-1},
           \end{aligned}
           $$
           where $\Sigma^\star \succ 0$ under the stated assumption. Further, since $\hat{\Sigma}, M$ are symmetric, $\Sigma^\star$ is symmetric as well. The second part of the result follows by substituting of $\tilde{V}^o$ into $J(\tilde{V})$ and rearranging the terms with algebraic manipulation.
           \end{proof}

    \section*{Appendix C: Proofs of Section \ref{sec:theory}}
    \subsection{Proof of Proposition \ref{KL:bound}}
    \begin{proof}
        Let $p(x)$ and $q(x)$ be the two density functions corresponding to $\mathbb{P}$ and $\mathbb{Q}$, respectively. In particular,
        $$
        q(x) = \frac{1}{\sqrt{(2\pi)^n \det \Sigma_q}} \exp \left( -\frac{1}{2} x^T \Sigma_q^{-1} x \right).
        $$
        Substituting the expression into the KL-divergence definition in \eqref{KL:density}, we get
        $$
        R(\mathbb{P} \| \mathbb{Q})  = \int p(x) \log p(x) \, dx - \int p(x) \log q(x) \, dx,
        $$
        where the frst term in the right-hand side is the differential entropy of $p(x)$, i.e.
        $
        h(p) = -\int p(x) \log p(x) \, dx.
        $
        Next, we compute the expectation of $\log q(x)$ under $p(x)$ as
    $$
        \mathbb{E}[\log q(x)] = -\frac{n}{2} \log (2\pi) - \frac{1}{2} \log \det \Sigma_q - \frac{1}{2} \mathbb{E}[x^T \Sigma_q^{-1} x].
     $$
        Since $\mathbb{E}[x x^T] = \Sigma_p$, and hence $
        \mathbb{E}[x^T \Sigma_q^{-1} x] = \text{tr}(\Sigma_q^{-1} \mathbb{E}[x x^T]) = \text{Tr}(\Sigma_q^{-1} \Sigma_p),$ we finally get
    $$
        \mathbb{E}[\log q(x)] = -\frac{n}{2} \log (2\pi) - \frac{1}{2} \log \det \Sigma_q - \frac{1}{2} \text{Tr}(\Sigma_q^{-1} \Sigma_p).
    $$
        Next, we provide an upper bound via a maximum entropy approach. Note that
        $$
            \begin{aligned}
        & h(p) \leq h(\mathcal{N}(0, \Sigma_p)) = \frac{n}{2} \log (2\pi e) + \frac{1}{2} \log \det \Sigma_p \iff \\
        & 		-h(p) \geq -h(\mathcal{N}(0, \Sigma_p)) = -\left(\frac{n}{2} \log (2\pi e) + \frac{1}{2} \log \det \Sigma_p\right)
            \end{aligned}
        $$
        where the first inequality is due to the fact that the Gaussian distribution minimizes the entropy for a fixed covariance.
        Using this upper bound yields
        $$
        \begin{aligned}
        & R(\mathbb{P} || \mathbb{Q}) = -h(p) - \mathbb{E}[\log q(x)]\\
        \geq & -\left( \frac{n}{2} \log (2\pi e) + \frac{1}{2} \log \det \Sigma_p \right) - \\
        & \left( -\frac{n}{2} \log (2\pi) - \frac{1}{2} \log \det \Sigma_q - \frac{1}{2} \text{Tr}(\Sigma_q^{-1} \Sigma_p) \right).
        \end{aligned}
      $$
        Simplifying the above expression concludes the proof.
        \end{proof}


        \subsection{Proof of Theorem \ref{thr:lin:gauss}}
        \begin{proof}
            We immediately recognize that \eqref{eq:reformulation:1} and \eqref{eq:lower:reformulation} are identical. Hence, (U) = (P) = (L). As a result, (P) is solved by a linear policy $\bu^\star = \tilde{\bU}^\star \boldeta + \tilde{\bq}^\star$ expressed in terms of the purified observations $\boldeta$ and a Gaussian distribution $\mathbb{P}^\star \in \mathcal{G}$. Finally, we invoke Lemma \ref{lemma:change} to conclude that $\bu^\star = \bU^\star \by + \bq^\star$ for $\bU^\star = (I + \tilde{\bU}^\star \bC \bH)^{-1} \tilde{\bU}^\star$ and $\bq^\star = (I + \tilde{\bU}^\star \bC \bH)^{-1} \tilde{\bq}^\star$.
            \end{proof}

    \section*{Appendix D: Proofs of Section \ref{Sec:computation}}
    \subsection{Proof of Proposition \ref{strong:duality}}
    \begin{proof}
        By linearity of the trace, the objective function in \eqref{eq:linear:policy} is concave (in fact, linear) in $\bW,\bV$; further, since $\bQ,\bR\succeq 0$ it is convex in $\bU$. Moreover, by Lemma \ref{G:nice},  $\mathcal{G}$ is a convex and compact set. Hence, we invoke Sion minimax theorem \cite{sion1958general} to conclude that strong duality holds.
        \end{proof}

    \subsection{Proof of Lemma \ref{set:S:compact}}
    \begin{proof}
        We show that $\mathcal{J}(\bU,\bSigma)$ is coercive in $\bU$. Consider the first component of $\mathcal{J}(\bU,\bSigma)$: We have
        $$
        \begin{aligned}
        & \text{Tr}(\bF_1(\bU) \bW) \\
        &= \text{Tr}(\bD^\top \bU^\top \bR \bU \bD \bW) + \text{Tr}((\bH \bU \bD + \bG)^\top \bQ(\bH \bU \bD + \bG) \bW)\\
        & \geq \lambda_{\min}(\bR) \lambda_{\min}(\bW) \|\bU \bD\|_F^2 +  \lambda_{\min}(\bQ) \lambda_{\min}(\bW) \|\bH \bU \bD + \bG\|_F^2\\
         & \geq \lambda_{\min}(\bR) \underline{\lambda}_W \|\bU \bD\|_F^2 + \lambda_{\min}(\bQ) \underline{\lambda}_W \|\bH \bU \bD + \bG\|_F^2
        \end{aligned}
        $$
        for any $\bU \in \mathbb{R}^{mT \times pT}$, where $\| \cdot \|_F$ is the Frobenius norm, and $\underline{\lambda}_W := \min_{\bW \in \mathcal{G}} \lambda_{\min} (\bW) > 0$ which exists and is finite as $\mathcal{G}$ is compact by Lemma~\ref{G:nice}. Similarly,
             $$
        \begin{aligned}
        & \text{Tr}(\bF_2(\bU) \bV) \\
        &= \text{Tr}(\bU^\top \bR \bU \bV) + \text{Tr}((\bU^\top \bH^\top \bQ \bH \bU)\bV)\\
        & \geq \lambda_{\min}(\bR) \lambda_{\min}(\bV) \|\bU\|_F^2 +  \lambda_{\min}(\bQ) \lambda_{\min}(\bV) \|\bH \bU\|_F^2\\
         & \geq \lambda_{\min}(\bR) \underline{\lambda}_V \|\bU\|_F^2 + \lambda_{\min}(\bQ) \underline{\lambda}_V \|\bH \bU\|_F^2,
        \end{aligned}
        $$
        where $\underline{\lambda}_V := \min_{\bV \in \mathcal{G}} \lambda_{\min} (\bV) > 0$ which again exists and is finite due to compactness of $\mathcal{G}$. Combining the above bounds, we conclude that, for all $\bSigma \in \mathcal{G}$, $\mathcal{J}(\bU,\bSigma)$ is lower-bounded by a quadratic form of $\|\bU\|_F$ with positive leading coefficient $c(\underline{\lambda}_W, \underline{\lambda}_V, \bQ, \bR, \bD, \bH, \bG) \in \mathbb{R}_{++}$ independent of $\bSigma$. Thus, $\mathcal{J}(\bU,\bSigma)$ is uniformly coercive in $\bU$ , i.e., for all $\alpha > 0$ there exists $K > 0$ such that $\|\bU\|_F \geq K \implies \mathcal{J}(\bU, \bSigma) \geq \alpha$ for all $\bSigma \in \mathcal{G}$. As a result, any minimizing sequence $\{\bU_k(\bSigma)\} \in \mathbb{R}^{mT \times pT}$ must be bounded for all $\bSigma \in \mathcal{G}$. By the Bolzano-Weierstrass theorem, it admits a convergent subsequence $\{\bU_{k_j}(\bSigma)\} \rightarrow \bU^\star(\bSigma)$ for all $\bSigma \in \mathcal{G}$. As $\mathcal{J}(\bU,\bSigma)$ is continuous, this ensures the existence of an optimal solution $\bU^\star(\bSigma)$.

        Let $\phi(\bSigma) := \min_{\bU \in \mathbb{R}^{mT \times pT}} \mathcal{J}(\bU, \bSigma)$ and notice that it is continuous since $\mathcal{J}$ is continuous and $\mathcal{G}$ compact. Define $M = \max_{\bSigma \in \mathcal{G}}\phi(\bSigma)$. Pick $\alpha = M + \epsilon$ with $\epsilon > 0$ arbitrarily small. By uniform coercivity, $\phi(\bSigma) \leq M < M + \epsilon = \alpha$. Thus, any minimizer $\bU^\star(\bSigma)$ must satisfy $\mathcal{J}(\bU^\star(\bSigma), \bSigma) < \alpha \implies \|\bU^\star(\bSigma)\|_F < K$. Hence, $\|\bU^\star(\bSigma)\|_F \leq K$ for all $\bSigma \in \mathcal{G}$. Thus, we can take $\mathcal{S} = \text{co}(\{ x \in  \mathbb{R}^{mT \times pT} : \|x\|_F \leq K \})$, where $\text{co}$ denotes the convex hull. By definition, $\mathcal{S}$ is compact, convex and contains all minimizers $\bU^\star(\bSigma)$ for all $\bSigma \in \mathcal{G}$.



        \end{proof}

        \subsection{Proof of Theorem \ref{thr:convergence:CT}}
        \begin{proof}
            The proof follows by invoking \cite[Theorem~1]{hofbauer2006best} after veryfing that all standing assumptions hold in our setting. Specifically, $\mathcal{J}(\bU,\bSigma)$ is a continuous saddle point function, and $\mathcal{S}, \mathcal{G}$ are compact and convex subsets of finite-dimensional Euclidean spaces.
            \end{proof}

    \section*{Appendix E: Proofs of Section \ref{sec:endo}}

    \subsection{Proof of Proposition \ref{propo:recursion}}
    We divide the proof into two parts. \\
    \textit{{First part: Optimal solutions.}}
    By exploiting \eqref{eq:tightness}, the definition of $H_{t+1}(I_{t+1})$ in \eqref{eq:opt:value:function} and the property of expectation of quadratic forms, we notice that the objective function in \eqref{eq:opt:value:function} is equivalent to
    \begin{equation} \label{eq:proof_H}
    \begin{aligned}
    & \inf_{u_t} \sup_{{\mu}_t, {W}_t\succ 0} {\mathbb{E}_{\mathbb{P}_{w_t}}}\left[ \frac{1}{2}(x_{t}^\top Q x_t + u_t^\top R u_t) + \right.\\
    & \left. \frac{1}{2} \left( Ax_t + Bu_t + {\mu}_t \right)^\top P_{t+1} \left( Ax_t + Bu_t + {\mu}_t \right)+   \right.\\
    & \left.  \frac{1}{2} \text{Tr}(P_{t+1}{W}_t) \frac{1}{2} \text{Tr}(S_{t+1} {\luci \Sigma_{t+1}} ) - \frac{\tau_t}{2} \left( \text{Tr}(\hat{W}_t^{-1}{W}_t) - n + \right. \right.\\
    & \left. \left.  {\mu}_t^\top \hat{W}_t^{-1}{\mu}_t +\ln \frac{|{\hat{W}_t}|}{|{W}_t|}  \right) \left| \right. I_{t}, u_t \right] + z_{t+1},
    \end{aligned}
    \end{equation}
    {\luci By standard arguments from filtering theory, it follows that $\Sigma_{t+1}$  can be expressed as a function of $\Sigma_{t}$ as}
    $$
    \Sigma_{t+1} ={W}_t + A \Sigma_t A^\top  - A \Sigma_t C^\top \left( C  \Sigma_t C^\top + \hat{V}_t \right)^{-1} C \Sigma_t A^\top.
    $$
    {\luci By substituting $ \Sigma_{t+1}$ in \eqref{eq:proof_H},  $H_t(I_{t})$ equals
    we obtain}
    \begin{equation}
    \begin{aligned} \textstyle
    & \inf_{u_t} \sup_{{\mu}_t, {W}_t\succ 0} {\mathbb{E}_{\mathbb{P}_{w_t}}}\left[ \frac{1}{2}(x_{t}^\top Q x_t + u_t^\top R u_t) + \right.
    \\
    & \left. \frac{1}{2} \left( Ax_t + Bu_t + {\mu}_t \right)^\top P_{t+1} \left( Ax_t + Bu_t + {\mu}_t 
    \right) +\right.
    \\
    & \left.
     +  \frac{1}{2} \text{Tr}(P_{t+1}{W}_t) +  \frac{1}{2} \text{Tr}(S_{t+1}{W}_t)  - \frac{\tau_t}{2} \left( \textstyle \text{Tr}(\hat{W}_t^{-1}{W}_t) - n + 
    \right. \right.
    \\
    & \left. \left.  {\mu}_t^\top \hat{W}_t^{-1}{\mu}_t + \ln \frac{|{\hat{W}_t}|}{|{W}_t|}  \right) \left| \right. I_{t}, u_t \right] + z_{t+1} +  \\
    & 
    \underbrace{\frac{1}{2}\text{Tr}\left( S_{t+1} (A  \Sigma_t A^\top \!-\!  A \Sigma_t C^\top \left( C \Sigma_t C^\top \!+\! \hat{V}_t\right)^{-1} \!C  \Sigma_t A^\top )\! \right)}_{r_t(\Sigma_t)}\!.
    \end{aligned}
    \end{equation}
    
    Notice that there are no terms coupling ${\mu}_t$ and ${W}_t$, hence we can carry out the two maximization problems separately. As for ${W}_t$, we obtain:
    \begin{equation}\label{V:start:opt}
    \begin{aligned}
    {W}_t^o & = \argmax_{{W}_t \succ 0} \frac{1}{2} \text{Tr}((P_{t+1} + S_{t+1}){W}_t) - \\
    & \frac{\tau_t}{2} \textstyle \left( \text{Tr}(\hat{W}_t^{-1}{W}_t) - n + \ln \frac{|{\hat{W}_t}|}{|{W}_t|}  \right).
    \end{aligned}
    \end{equation}
    Since \eqref{eq:tau:conditionPS} holds,
    we can invoke Lemma \ref{lemma:opt:variance} to conclude that
    $$
    {W}_t^o = \textstyle \left( \hat{W}_t^{-1} - \frac{(P_{t+1} + S_{t+1})}{\tau_t}  \right)^{-1}. 
    $$
    
    As for ${\mu}_t$, we obtain:
    $$
    \begin{aligned}
   {\mu}^o_t & = \argmax_{
    {\mu}_t} {\mathbb{E}_{\mathbb{P}_{w_t}}} \left[  \frac{1}{2} {\mu}_t^\top P_{t+1}{\mu}_t +  u_t^\top B^\top P_{t+1} {\mu}_t + \right.\\
    & \left.  x_t^\top A^\top P_{t+1} {\mu}_t - \frac{\tau_t}{2} {\mu}_t^\top V^{-1} {\mu}_t | I_{t}, u_t\right].
    \end{aligned}
    $$
    Under \eqref{eq:tau:conditionPS},
    the objective function in the last equation is concave in ${\mu}_t. $ 
    By \cite{boyd:vandenberghe:2004}, the maximum is obtained by
    annihilating the first derivative, i.e. by imposing
    $
        P_{t+1} {\mu}_t + P_{t+1} B u_t + P_{t+1} A \hat{x}_t - \tau_t \hat{W}_t^{-1} {\mu}_t = 0 , 
    $
    which leads to
    \begin{equation}\label{eq:mu:start}
    \begin{aligned}
    {\mu}^o_t & = \left( \tau_t \hat{W}_t^{-1} - P_{t+1}\right)^{-1} P_{t+1} \left(A\tilde{x}_t + Bu_t \right)\\
    & = \textstyle\left[ \left( I - \frac{\hat{W}_t P_{t+1}}{\tau_t}\right)^{-1} - I \right]\left(A\tilde{x}_t + Bu_t \right).
    \end{aligned}
    \end{equation}
    
    For the outer minimization problem with respect to $u_t$, we first notice that ${\mu}^o_t$ is a function of $u_t$. 
    Then, grouping the terms that depend on $u_t,$ we have that 
    \begin{equation}\label{eq:u:minimization}
    \begin{aligned}
    {u}^o_t & = \argmin_{u_t}
    {\mathbb{E}_{\mathbb{P}_{w_t}}} \left[
    \frac{1}{2} u_t^\top R u_t + 
    \frac{1}{2} u_t^\top B^\top P_{t+1} B u_t  + \right. \\
    & \left.   
    \frac{1}{2} 
    ({\mu}^{o}_t)^\top P_{t+1}{\mu}^o_t + 
    x_t^\top A^\top P_{t+1} B u_t + 
    u_t^\top B^\top P_{t+1} {\mu}^{o}_t 
    + \right. \\
    & \left. 
    x_t^\top A^\top P_{t+1} {\mu}^{o}_t -
    \frac{\tau_t}{2} ({\mu}^{o}_t)^\top \hat{W}_t^{-1} {\mu}^o_t 
    | I_{t} \right].
    \end{aligned}
    \end{equation}
    From \eqref{eq:mu:start} we have that 
    $ {\mu}^o_t = \frac{1}{\tau_t} \hat{W}_t P_{t+1} \left( A \tilde{x}_t + B u_t + {\mu}^o_t \right),$
    from which it follows that the derivative of the objective function in \eqref{eq:u:minimization} with respect to $u_t$ is
    $$
    \begin{aligned}
    & {\mathbb{E}_{\mathbb{P}_{w_t}}}\left[ R u_t + B^\top P_{t+1} \left( A x_t + B u_t + {\mu}^o_t \right) +  \right.\\
    & \left. \frac{\partial {\mu}_t^o}{\partial u_t}^\top P_{t+1} \left( A \tilde{x}_t + B u_t + {\mu}^o_t \right) -  \tau_t \frac{\partial {\mu}_t^o}{\partial u_t}^\top \hat{W}_t^{-1} {\mu}_t^o \right] \\
    = &  R u_t + B^\top P_{t+1} \left( A \tilde{x}_t + B u_t + {\mu}^o_t \right) \\
    = & R u_t + B^\top P_{t+1} \left( I - \frac{\hat{W}_t P_{t+1}}{\tau_t}\right)^{-1} (A \tilde{x}_t +  B u_t).
    \end{aligned}
    $$
    Consequently, the second derivative with respect to $u_t$ 
    of the objective function in \eqref{eq:u:minimization} is 
    ${\luci R + B^\top  \big( P_{t+1}^{-1} - \frac{V}{\tau_t} \big)^{-1}  B },$ which is positive semidefinite under assumption  \eqref{eq:tau:conditionPS}.
    We conclude that the objective function in \eqref{eq:u:minimization} is convex with respect to $u_t$, hence the optimal control input $u_t^o$ can be derived by annihilating the first derivative. Applying Woodbury matrix identity, we get
    \begin{equation} \label{eq:utstar_proof}
    \begin{aligned}
    u_t^o & = -R^{-1} B^\top \textstyle \left( P_{t+1}^{-1} + BR^{-1}B^\top - \frac{\hat{W}_t}{\tau_t} \right)^{-1}A \tilde{x}_t.
    \end{aligned}
    \end{equation}

    \textit{{Second part: Optimal value function.}}
    Next, we derive a closed-form expression for the DP recursions. By recalling the definition of $r_t(\Sigma_t)$ and rearranging some terms, we express the optimal objective function in \eqref{eq:opt:value:function} as
    \begin{equation}\label{eq:H:init}
    \begin{aligned}
    & H_t(I_{t})  \coloneqq 
     \underbrace{ {\mathbb{E}_{\mathbb{P}_{w_t}}}
    \left[\frac{1}{2}(x_{t}^\top Q x_t + u_t^{o\top} R u^o_t)\right. }_{(\spadesuit)} +\\
    & \underbrace{ \left. \frac{1}{2} \!\left( \!Ax_t \!+\! Bu^o_t \!+\! {\mu}^o_t \!\right)^\top\!\! P_{t+1} \!\left(\! Ax_t \!+\! Bu^o_t \!+ \!{\mu}^o_t\! \right) \!- \!\frac{\tau_t}{2}{\mu}_t^{o\top} \! \hat{W}_t^{-1}{\mu}^o_t\right. \! | I_t] }_{(\spadesuit)}\\ 
    & \left. 
    + \underbrace{
     \frac{1}{2} \text{Tr}(( P_{t+1} \!+ \!S_{t+1} ) {W}^o_t) \!-\! \frac{\tau_t}{2} \big( \text{Tr}(\hat{W}_t^{-1}{W}^o_t) \!-\! n  \!+ \! \ln \frac{|{\hat{W}_t}|}{|{W}^o_t|} \big)}_{(\clubsuit)}  \right.\\
    & \left.
     + z_{t+1} + r_t(\Sigma_t) | I_{t} \right] .
    \end{aligned}
    \end{equation}
    From Lemma  \ref{lemma:opt:variance}, $
    (\small{\clubsuit}) = -\frac{\tau_t}{2}\log \left| I - \frac{\hat{W}_t(P_{t+1} + S_{t+1})}{\tau_t}\right|.
    $
    Moreover, 
    {\allowdisplaybreaks
    \begin{align*}
    (\spadesuit) &=  \frac{1}{2}\tilde{x}_t^\top Q \tilde{x}_t + \frac{1}{2}\text{Tr}(\Sigma_t Q) + \frac{1}{2}u_t^{o \top}R u_t^o + \\
    & \frac{1}{2}\left( A\tilde{x}_t + Bu^o_t + {\mu}^o_t \right)^\top P_{t+1} \left( A\tilde{x}_t + Bu^o_t + {\mu}^o_t \right) + \\
    &  + \frac{1}{2} \text{Tr}( A^\top P_{t+1} A\Sigma_t) -\frac{\tau_t}{2}{\mu}_t^{o \top}\hat{W}_t^{-1}{\mu}_t^o \\
    = & \frac{1}{2}\tilde{x}_t^\top Q \tilde{x}_t +  
    \frac{1}{2}\text{Tr}(\Sigma_t (Q +  A^\top P_{t+1}A ))+ \frac{1}{2}u_t^{o \top}R u_t^o + \\
    & \frac{1}{2}\!\left( A\tilde{x}_t\! + \!Bu^o_t\right)^\top\! \textstyle\left( I\! -\! \frac{\hat{W}_t P_{t+1}}{\tau_t} \right)^{-1} \!P_{t+1}\textstyle\left( I \!- \!\frac{\hat{W}_t P_{t+1}}{\tau_t} \right)^{-1}\! \cdot \\
    & ( A\tilde{x}_t \!+\! Bu^o_t)
     - \frac{\tau_t}{2} ( A\tilde{x}_t\! + \!Bu^o_t )^\top\! \Big( \big( I \!- \!\frac{\hat{W}_t P_{t+1}}{\tau_t} \big)^{-1} \!-\! I \Big) \cdot \\ 
     & \hat{W}_t^{-1} \Big( \textstyle \big( I \!- \!\frac{\hat{W}_t P_{t+1}}{\tau_t} \big)^{-1} \!- \!I \Big)( A\tilde{x}_t\! + \!Bu^o_t )\\
     = & \frac{1}{2}\tilde{x}_t^\top Q \tilde{x}_t + 
     \frac{1}{2}\text{Tr}(\Sigma_t (Q +  A^\top P_{t+1}A)) + \frac{1}{2}u_t^{o \top}R u_t^o +\\
     & \frac{1}{2}\left( A\tilde{x}_t\! + \!Bu^o_t\right)^\top M \left( A\tilde{x}_t\! + \!Bu^o_t\right),
    \end{align*}
    }
    {\luci where the second equality follows from \eqref{eq:mu:start}. Here}
    $$
    \begin{aligned}
    M & = \textstyle\left[ \textstyle\left( I\! -\! \frac{\hat{W}_t P_{t+1}}{\tau_t} \right)^{-1} \!P_{t+1}\left( I \!- \!\frac{\hat{W}_t P_{t+1}}{\tau_t} \right)^{-1} \right.\\
     & \left. - \tau_t \! \textstyle\left( \textstyle\left( \!I \!- \!\frac{\hat{W}_t P_{t+1}}{\tau_t} \right)^{-1} \!- \!I \! \right)\hat{W}_t^{-1} \textstyle\left( \!\textstyle\left( I \!- \!\frac{\hat{W}_t P_{t+1}}{\tau_t} \right)^{-1} \!- \!I \right) \right]\\
     & = \textstyle\left(I - \frac{P_{t+1}\hat{W}_t}{\tau}\right)^{-1}P_{t+1},
    \end{aligned}
    $$ 
    with the second equality following from algebraic manipulations. Hence, 
    {\allowdisplaybreaks
    \begin{align*}
    (\spadesuit) = &  \frac{1}{2}\tilde{x}_t^\top Q \tilde{x}_t + \frac{1}{2}\text{Tr}(\Sigma_t (Q +  A^\top P_{t+1}A )) 
    + \frac{1}{2}u_t^{o \top}R u_t^o +\\
    & \frac{1}{2}\left( A\tilde{x}_t\! + \!Bu^o_t\right)^\top M \left( A\tilde{x}_t\! + \!Bu^o_t\right)\\
    = & \frac{1}{2}\tilde{x}_t^\top \left(Q   + A^\top M A\right) \tilde{x}_t + \frac{1}{2}u_t^{\top o} \left(R   + B^\top M B\right) u_t^o + \\
    & u_t^{o \top} B^\top MA \tilde{x}_t + \frac{1}{2} \text{Tr}\left(\Sigma_t\left( Q + A^\top P_{t+1}A \right)  \right)\\
      \overset{(a)}{=} & \frac{1}{2}\tilde{x}_t^\top \left(Q   + A^\top M A\right) \tilde{x}_t + \frac{1}{2} \text{Tr}\left(\Sigma_t\left( Q + A^\top P_{t+1}A \right)  \right)\\
     & - \frac{1}{2}\tilde{x}_t^\top A^\top M B \left(  R + B^\top M B\right)^{-1} B^\top M A \tilde{x}_t \\
      = & \frac{1}{2}\tilde{x}_t^\top \!\left(Q  \! + \!A^\top M A \! +\! A^\top M B \!\left(  R \!+\! B^\top M \!B\right)^{-1} \!\!B^\top M \!A\!\right) \!\tilde{x}_t  \\
       & + \frac{1}{2} \text{Tr}\left(\Sigma_t\left( Q + A^\top P_{t+1}A \right)  \right)\\
    \overset{(b)}{=} & \frac{1}{2}\tilde{x}_t^\top \left(Q   + A^\top \textstyle \left( P_{t+1}^{-1} + B R^{-1}B^\top - \frac{\hat{W}_t}{\tau_t}\right)^{-1} A \right) \tilde{x}_t + \\
       & \frac{1}{2} \text{Tr}\left( \Sigma_t\left( Q + A^\top P_{t+1}A \right)  \right)\\
    \overset{(c)}{=} &  \frac{1}{2} \tilde x_t^\top P_t \tilde x_t +\frac{1}{2} \text{Tr}\left( \Sigma_t\left( Q + A^\top P_{t+1}A \right)  \right), 
    \end{align*}
    }
    {\luci where  $(a)$ follows from substituting $u_t^o = -(R+B^\top M B)^{-1}B^\top M A \tilde{x}_t$ derived in \eqref{eq:utstar_proof}, }
    $(b)$ from Woodbury matrix identity and $(c)$ from defining $P_t = Q   + A^\top \left( P_{t+1}^{-1} + B R^{-1}B^\top - \frac{\hat{W}_t}{\tau_t}\right)^{-1}A.$ Then, combining $(\clubsuit)$ and $(\spadesuit)$, and exploiting the expression for the expectation of quadratic forms, \eqref{eq:H:init} becomes
    {\allowdisplaybreaks
    \begin{align*}
     & \frac{1}{2} \tilde x_t^\top P_t \tilde x_t +\frac{1}{2} \text{Tr}\left( \Sigma_t\left( Q + A^\top P_{t+1}A \right)  \right) - \\
     & \frac{\tau_t}{2}\log \left| I - \frac{\hat{W}_t(P_{t+1} + S_{t+1})}{\tau_t}\right| + z_{t+1}  +r_t\\
       \overset{(d)}{=} & {\mathbb{E}}\left[\frac{1}{2} x_t^\top P_t x_t + \frac{1}{2}\xi_t^\top \left( Q + A^\top P_{t+1}A - P_t \right)\xi_t  |I_t \right]   \\
        & -\frac{\tau_t}{2}\log \left| I - \frac{\hat{W}_t(P_{t+1} + S_{t+1})}{\tau_t}\right| + z_{t+1} + r_t,
    \end{align*}
    }
    where in $(d)$ we recall that $\xi_t = x_t - \tilde{x}_t$ is the estimation error.
    The thesis follows by defining $S_t = Q + A^\top P_{t+1}A - P_t $ and $z_t = -\frac{\tau_t}{2}\log \left| I - \frac{\hat{W}_t(P_{t+1} + S_{t+1})}{\tau_t}\right| + z_{t+1} . \hfill \square$

    \subsection{Proof of Theorem \ref{CGD:convergence}}
    \begin{proof}
        We begin by showing that $\bar \Wc_0(I_0, \tau) $ is jointly convex in $\tau = (\tau_0, \dots, \tau_{T-1})$ for any information vector $I_0.$ First, we prove that the statement is true when the state-transition matrix $A$ is invertible. 
        Consider {\agu any sequence of positive semidefinite matrices  
        $ \{ \Lambda_t \in \mathbb{S}_{+}^n \}_{t=0}^{T-1} $.
        In particular, we are interested in the sequence defined by the following iteration, but this will become essential only after equation \eqref{eq:conv_proof_mean0}:} 
        \begin{equation}\label{eq:Lambda_rec}
             \Lambda_{t+1} = 
             A M_t \hat{\Sigma}_t^{-1} \Lambda_{t} \hat{\Sigma}_t^{-1} M_t A^\top 
             + A M_{t} C^{\top} \hat{V}_t^{-1} C  M_{t} A^\top 
             + {W}_t
        \end{equation}
        with $ \Lambda_0 = \hat{\Sigma}_0.$ 
        Define the linear system
        $$
        \omega_{t+1} = A \omega_{t} + B u_t + {w}_t
        $$
        with {\agu $\omega_0$ being a random vector with mean
        equal to $\bar{\omega}_0$ and covariance matrix
        equal to $\Omega_0$.
        We call ${f}_{-1}$ the probability density  
        of $\omega_0$ and assume that $\omega_0$ is independent of ${w}_t$ for all $t=0,1,\dots, T$.}
        Let $U_t := [u_0, \dots, u_t] $ be the collection of input vectors up to time $t,$
        and define  
        $$ \Omega_{t+1} :=  E \big[ 
        (\omega_{t+1} - E[\omega_{t+1}| U_t]) 
        (\omega_{t+1} -  E[\omega_{t+1}| U_t)^\top 
        | U_t \big].$$
        Clearly, it holds that 
        \begin{equation} \label{eq:omega_rec} 
        \Omega_{t+1} = A \Omega_{t} A^\top + {W}_t
        \end{equation}
        for $t=0, \dots, N.$ 
        {\agu As a consequence of  the invertibility of the matrix A, we can select $\Omega_0 \in \mathbb{S}_{+}$ in such a way} that 
        \begin{equation}\label{eq:lambda_omega_ineq}
             \Omega_t \geq \Lambda_t,  \quad t = 0, \dots, T. 
        \end{equation}
        In view of \eqref{eq:lambda_omega_ineq}, there exists a sequence of random vectors
        $\{ e_t \}_{t=0}^{T}$ 
        such that, 
        letting $ \hat{\omega}_{t+1} := {\mathbb{E}}\left[ \omega_{t+1} | U_{t}, e_{t+1} \right],$ we have
        \begin{align}
        & {\mathbb{E}}
        \left[ (\omega_{t+1} - \hat{\omega}_{t+1} ) 
        (\omega_{t+1} - \hat{\omega}_{t+1})^\top 
        | U_{t}, e_{t+1}
        \right] = \Lambda_{t+1},  \label{eq:conv_proof_cov_t}  \\
        & {\mathbb{E}}
        \left[ (\omega_{0} - \hat{\omega}_{0} ) 
        (\omega_{0} - \hat{\omega}_{0})^\top 
        | e_{0}
        \right] = \Sigma_{0}, \label{eq:conv_proof_cov0} \\
        & {\mathbb{E}} \left[ \omega_0 | e_0 \right] = \hat{x}_0. \label{eq:conv_proof_mean0}
        \end{align}
        Let $\tilde{I}_0 = e_0 $ and $\tilde{I}_t = (U_{t-1}, e_{t})  $ for $t=1,\dots, T-1,$ 
        and consider 
        $u_t = \tilde{\pi}_t(\tilde{I}_t) $
        and ${f}_t = \tilde{\gamma}_t (\tilde{I}_t) $ 
        where $ \tilde{\pi}_t $ and  $\tilde {f}_t $ are measurable functions which map the ``information vector'' $ \tilde I_t$ to a control input $u_t$ and a density measure $ {f}_t,$ respectively,  at each time step.
        Now, consider the new control problem
        \begin{multline*} 
        \mathcal{Q}_0 (\tilde{I}_0, \tau) := \textstyle
        \inf_{ \big\{u_t = \tilde{\pi}_t (\tilde{I}_t) \big\}_{t=0}^{N} } 
        \; 
        \sup_{
        \big\{{f}_t = \tilde{\gamma}_t (\tilde{I}_t)  \in \mathcal{P}_{\mathcal{G}} (\mathbb{R}^n) \big \}_{t=0}^{T-1} 
        }
        \\ 
        {\mathbb{E}} \Big[ \sum_{t=0}^{T+1} \Big(\frac{1}{2} (\omega_t^\top Q \omega_t + u_t^\top R u_t) - \tau_t \mathcal{R}_t \Big)  + \frac{1}{2} x_{T}^\top Q_{T} x_{T}\Big].
        \end{multline*} 
        With computations similar to Subsection \ref{sec:approx}, exploiting \eqref{eq:Lambda_rec}, \eqref{eq:conv_proof_cov_t}, \eqref{eq:conv_proof_cov0}, \eqref{eq:conv_proof_mean0}, we conclude that $
        \mathcal{Q}_0 (\tilde{I}_0, \tau) = \bar{\mathcal{V}}_0 (I_0, \tau) 
        $
        for any fixed $\tau. $
        Note that $\mathcal{Q}_0 (\tilde{I}_0, \tau)$
        is jointly convex in $\left(u_0, \dots, u_N, \tau \right)$ for any $(\tilde f_0, \dots, f_{T-1})$. By Lemma \ref{lemm:convexity_max}, the maximization with respect to $( f_0, \dots,f_{T-1})$ yields a jointly convex function in $\left(u_0, \dots, u_{T-1}, \tau \right)$. Moreover, in view of Lemma \ref{lemm:convexity_min}, the partial minimization with respect to the input variables $\left(u_0, \dots, u_{T-1} \right)$ preserves convexity with respect to $\tau$.
        We therefore conclude that $\mathcal{Q}_0 (\tilde{I}_0, \tau)$, hence $\bar{\mathcal{V}}_0 (I_0, \tau),$ is convex in $\tau$. 
        Consequently, $ \bar{\mathcal{W}}_0 (I_0, \tau)$ in \eqref{pb:dual_approx} is convex in $\tau$ as sum of two convex functions. 
        
        Finally, we extend this result by continuity to the case where the state-transition matrix $A$ is singular. Let \( A \) be a singular matrix. {\agu We can select a positive and arbitrarily small \( \epsilon \in \mathbb{R} \) such that} \( \tilde{A} = A + \epsilon I \) is invertible. Since the function \( \bar{\mathcal{W}}_0 (I_0, \tau) \) is convex in \( \tau \) for \( \tilde{A} \) and continuous with respect to \( \epsilon \), we conclude that taking the limit as \( \epsilon \to 0 \) implies \( \bar{\mathcal{W}}_0 (I_0, \tau) \) remains convex in \( \tau \) for \( A \) as well.
     
        At this point, invoking \cite[Chapter~9]{nocedal1999numerical}, which holds since $\Wc_0 (I_0, \tau)$ is jointly convex in $\tau$ for any $I_0$, conclude the proof.
        \end{proof}
    

\end{document}